\documentclass{article}

 \usepackage{url}
\usepackage{xcolor}
\usepackage{makecell}
\usepackage{booktabs}
\usepackage{fullpage}
\usepackage{amsthm}
\usepackage{amsmath,amsfonts,multirow}
\usepackage[capitalize]{cleveref}
\usepackage{bbm}
\usepackage{algorithm} 
\usepackage{algpseudocode} 
\usepackage{graphicx} 
\usepackage{csquotes} 
\usepackage{placeins} 
\usepackage{subcaption}

\newtheorem{theorem}{Theorem}[section]
\newtheorem{lemma}[theorem]{Lemma}
\newtheorem{corollary}[theorem]{Corollary}

\newtheorem{proposition}[theorem]{Proposition}

\Crefname{assumption}{Assumption}{Assumptions}

\newtheorem{definition}[theorem]{Definition}

\providecommand{\keywords}[1]
{
  \small	
  \textbf{\textit{Keywords---}} #1
}

\newcommand{\edit}[1]{\textcolor{black}{#1}}

\newcommand{\E}{\Embb}

\newcommand{\R}{\Rmbb}

\newcommand{\Ek}{\E_k}
\newcommand{\Einit}{\E_0}

\newcommand{\Ekb}[1]{\Ek\bracks{#1}}

\newcommand{\parens}[1]{\left( #1 \right)}

\newcommand{\bracks}[1]{\left[ #1 \right]}

\newcommand{\norm}[1]{\left\| #1 \right\|}
\newcommand{\normsq}[1]{\norm{#1}^2}
\newcommand{\abs}[1]{\left| #1 \right|}
\newcommand{\abssq}[1]{\left| #1 \right|^2}

\newcommand{\qqtext}[1]{\qquad \text{#1} \qquad}

\newcommand{\thup}{^{\text{th}}}

\DeclareMathOperator{\row}      {row}

\DeclareMathOperator{\Diag}     {Diag}

\newcommand{\cM}{{\cal M}}

\newcommand{\mA}{\mathbf{A}}
\newcommand{\mP}{\mathbf{P}}
\newcommand{\mD}{\mathbf{D}}
\newcommand{\mQ}{\mathbf{Q}}
\newcommand{\mI}{\mathbf{I}}
\newcommand{\mG}{\mathbf{G}}
\newcommand{\vv}{\mathbf{v}}
\newcommand{\vx}{\mathbf{x}}
\newcommand{\vb}{\mathbf{b}}

\newcommand{\Rmbb}{\mathbb{R}}
\newcommand{\Embb}{\mathbb{E}}

\newcommand{\ik}{{i_k}}

\newcommand{\indepset}{\cM}

\newcommand{\xinit}{{\vx^0}}
\newcommand{\xone}{{\vx^1}}
\newcommand{\xk}{{\vx^k}}
\newcommand{\xkpo}{{\vx^{k+1}}}
\newcommand{\xopt}{{\vx^\star}}

\newcommand{\vbi}{\vb_i}
\newcommand{\vbj}{b_j}

\newcommand{\mAi}{\mA_i}

\newcommand{\mAik}{\mA_{i_k}}
\newcommand{\vbik}{b_{i_k}}

\newcommand{\nrows}{m}
\newcommand{\ncols}{n}

\newcommand{\selectset}{\mathcal{S}
}

\begin{document}

\title{Selectable Set Randomized Kaczmarz}

\author{Yotam Yaniv\footnote{Corresponding author, email: yotamya@math.ucla.edu} \footnote{University of California, Los Angeles, Department of Mathematics} , Jacob D. Moorman\footnotemark[\value{footnote}] , William Swartworth\footnotemark[\value{footnote}] ,\\ Thomas Tu\footnotemark[\value{footnote}] , Daji Landis\footnotemark[\value{footnote}] , Deanna Needell\footnotemark[\value{footnote}]}

\maketitle

\abstract{
The Randomized Kaczmarz method (RK) is a stochastic iterative method for solving linear systems that has recently grown in popularity due to its speed and low memory requirement.
Selectable Set Randomized Kaczmarz (SSRK) is an variant of RK that leverages existing information about the Kaczmarz iterate to identify an adaptive ``selectable set'' and thus yields an improved convergence guarantee. 
In this paper, we propose a general perspective for selectable set approaches and prove a convergence result for that framework. 
In addition, we define two specific selectable set sampling strategies that have competitive convergence guarantees to those of other variants of RK.
One selectable set sampling strategy leverages information about the previous iterate, while the other leverages the orthogonality structure of the problem via the Gramian matrix. We complement our theoretical results with numerical experiments that compare our proposed rules with those existing in the literature.
}
\\\keywords{Least norm solution, Stochastic iterative method, Kaczmarz method, Selectable set, Adaptive sampling}

\section{Introduction}

The Kaczmarz method \cite{karczmarz1937}, also known as the algebraic reconstruction technique in computed tomography \cite{ART1970}, has become a popular method for solving large overdetermined systems of linear equations. The method has abundant applications ranging from digital signal and image processing to statistics and machine learning. We are primarily interested in the the regime of extremely large linear systems, where it may be too expensive to load a large number of rows into memory. In this setting, the Kaczmarz method is particularly useful as it only requires loading a single row into memory at a time.  We also consider sparse systems, which yield additional benefits for the Kaczmarz method; the time required for each iteration scales linearly with the number of nonzero entries in the selected row \cite{natterer2001mathematics}.
 
To solve a system of equations $\mA \vx = \vb$ with $\mA \in \R^{\nrows \times \ncols}$, the Kaczmarz method operates iteratively, beginning with an initial vector $\xinit$ (often $\xinit = \mathbf{0}$).  On each iteration $k$, an equation $\mAik \vx = \vbik$, or equivalently row index $i_k$, is chosen and $\xkpo$ is computed as the projection of $\xk$ onto the set of solutions to that equation. Algebraically, the Kaczmarz update is given by 
\begin{equation}\label{eqn:kacz-update}
    \xkpo = \xk - \frac{\mAik \xk - b_{i_k}}{\normsq{\mAik}}\mAik^T,
\end{equation}
where $i_k$ is the index of the chosen equation, $\mAik$ is the corresponding row of the matrix $\mA$, and $\norm{\cdot}$ is the Euclidean norm. 
Algebraically projecting onto the solution of an equation $\mAik \vx = \vbik$ is equivalent to sampling a row index $i_k$ and applying \Cref{eqn:kacz-update}. Thus we refer to sampling then projecting onto \emph{equations} and sampling then applying \Cref{eqn:kacz-update} to \emph{row indices} interchangeably.

Like many iterative methods, the Kaczmarz method utilizes and depends on a sampling strategy to choose the equation for its update at each iteration.
Different sampling strategies exhibit different convergence behavior.
The first sampling strategy proven to result in linear convergence was a randomized strategy where equations are chosen at random with probabilities proportional to the corresponding squared row norms $\normsq{\mAi}$, which we deem the Randomized Kaczmarz method (RK) \cite{strohmer2009randomized} .
Subsequently, many variants of RK have been proposed and shown to converge linearly \cite{bai2018relaxed,nutini2016convergence,du2019tight,gower2015randomized,haddock2021greed,moorman2021randomized}. 

In this paper, we aim to develop Kaczmarz methods with linear convergence rates that mitigate inefficiencies in classical Kaczmarz methods by leveraging meta-information about the algorithm or problem.
For instance, if the equation chosen on iteration $k$ is already solved (i.e., $\mAik \xk = \vbik$), then \Cref{eqn:kacz-update} reduces to $\xkpo = \xk$ and the iteration is wasted. Therefore, it is desirable avoid sampling equations that are already solved by the current iterate $\xk$. In general, checking whether an equation is solved is as expensive as the update itself. However, if the system of equations $\mA \vx = \vb$ has some structure such as a known Gramian matrix $\mG = \mA \mA^T$, it can be possible to keep track of some equations that are known to be solved so that they can be avoided.
The set of equations that are not known to be solved is referred to as the \emph{selectable set} \cite{nutini2016convergence}.
In this work, we consider variants of RK that use a selectable set to avoid wasting iterations. Such variants are referred to as Selectable Set Randomized Kaczmarz methods (SSRK).

\subsection{Related Work}

Several works have considered using more general distributions in RK and have obtained similar convergence guarantees (see e.g.  \cite{gower2015randomized,needell2016stochastic,nutini2016convergence} and references therein).
A different line of work has focused on sampling strategies that depend on the iterate $\xk$ and thus change from iteration to iteration \cite{bai2018relaxed,nutini2016convergence,du2019tight,haddock2021greed}.  Most notable in the latter is the Max-Distance Kaczmarz method (MDK), also known as Motzkin's method, which chooses the equation that maximizes the normalized residual $\abs{\mAik \xk - \vbik}/\norm{\mAik}$ on each iteration.
The term max-distance refers to the fact that MDK chooses the equation that leads to the largest update, since
\begin{equation*}
    \frac{\abs{\mAik \xk - \vbik}}{\norm{\mAik}} = \norm{\xkpo-\xk}.
\end{equation*}
MDK yields a provably optimal per-iteration convergence guarantee at the expense of a high per-iteration computational cost \cite{nutini2016convergence}.

Several sampling strategies have been proposed that approximate MDK with a cheaper per-iteration cost.  For example, the Sampling Kaczmarz Motzkin method (SKM) chooses a random subset of rows and selects the maximum-residual row from that subset \cite{de2017sampling}.  This results in much cheaper per-iteration costs than MDK, while still yielding a provably better convergence guarantee than RK \cite{haddock2021greed}. 
Similarly to MDK, the Relaxed Greedy Randomized Kaczmarz method (RGRK) samples  from the equations  whose normalized residual $\abs{\mAik \xk - \vbik}/\norm{\mAik}$ exceeds some threshold \cite{bai2018greedy, bai2018relaxed}. RGRK has a faster convergence guarantee than RK but a slower guarantee than MDK and is significantly more expensive per-iteration than MDK \cite{bai2018relaxed,gower2021adaptive}. We compare the convergence guarantee of RGRK \cite{bai2018relaxed} with that of SSRK in \Cref{sec:bai-wu-compare}.

Nutini et al. consider several improvements to RK for sparse systems \cite{nutini2016convergence}.
In particular, this work introduces leveraging the \emph{orthogonality graph}, which is formed from by considering the Gramian matrix $\mG = \mA \mA^T$ as an adjacency matrix for an unweighted graph.
In this graph the nodes are the rows of the matrix.
Two nodes are joined by an edge if they are \emph{not orthogonal}.
The key observation is that a Kaczmarz update for row $\mA_i$ only affects residual entries corresponding to adjacent nodes in this graph.
This allows for tracking a so-called selectable set, which is the complement of the set of nodes for which the corresponding residual entry is known to be $0$.
Since sampling non-selectable rows yields no progress, one could hope to speed up RK by restricting the sampling to the selectable set.
We consider the selectable set method in more detail, and specialize to several common types of sparse orthogonality graphs.

\subsection{Contribution}
In this paper, we define a general framework for what we call \emph{Selectable Set Randomized Kaczmarz methods (SSRK,
\Cref{algo:selectablesetrk})}. This is a generalization of the orthogonality graph method proposed by Nutini et al. \cite{nutini2016convergence}. 
We show that SSRK methods converge linearly with a speedup over RK related to the size of the selectable set.
We define and analyze two specific SSRK methods, the Non-Repetitive Selectable Set Randomized Kaczmarz method (NSSRK, \Cref{algo:nssrk}) and the Gramian Selectable Set Randomized Kaczmarz method (GSSRK, \Cref{algo:gssrk}). These methods use different strategies to identify the selectable set. We show that NSSRK has a selectable set of size $m-1$, while the size of the GSSRK selectable set is bounded from below by properties of the matrix $\mA$. Finally, we note that the convergence guarantee of NSSRK is the same as that of Relaxed Greedy Randomized Kaczmarz method (RGRK) \cite{bai2018greedy, bai2018relaxed} despite converging much slower than RGRK in practice. This suggests that the convergence guarantee for RGRK is not tight.

\subsection{Organization}
The rest of this paper is structured as follows. The remainder of this section summarizes the notation that will be used throughout. 
In \Cref{sec:selectset} we define SSRK methods and define two specific examples, NSSRK and GSSRK.
Then, in \Cref{sec:convergence}, we prove a general convergence guarantee for SSRK methods, presented in \Cref{thm:selectable-one-step}, and use it to prove corollaries for specific methods and sampling strategies. Additionally, we discuss connections between the convergence analysis of \Cref{algo:nssrk} and a popular Kaczmarz method proposed by Bai and Wu \cite{bai2018greedy,bai2018relaxed}. Next, in \Cref{sec:lower-bd-gssrk}, we examine the improvement of applying \Cref{algo:gssrk} to problems with structured systems. Then, we show some empirical results in \Cref{sec:experiments} and finally, in \Cref{sec:conclusion}, we summarize our work and provide a short discussion on Kaczmarz sampling strategies. 

\subsection{Notation and Assumption}

We consider consistent systems of linear equations $\mA \vx = \vb$ with $\mA \in \R^{\nrows \times \ncols}$,  $\vx \in \R^{\ncols}$, and right hand side vector (RHS) $\vb \in \R^{\nrows}$.
We seek the least-norm solution to the system $\xopt=\mA^\dagger\vb$.
Throughout this paper, $m$ will globally represent the number of rows in the system and $n$ will represent the number of columns of the matrix $\mA$.  
Bold uppercase letters represent matrices, bold lowercase letters represent vectors, and standard letters represent scalars.
$\mAi$ denotes the $i\thup$ row of the matrix $\mA$, while $\vbi$ denotes the $i\thup$ element of the vector $\vb$.
We use $[m]$ as shorthand for the set $\{1, 2, \ldots, \nrows\}$.
The norm $\norm{\cdot}$ is the Euclidean vector norm and $\norm{\cdot}_F$ is the Fr\"obenius matrix norm.
The smallest \emph{nonzero} singular value of $\mA$ is $\sigma_{\min}(\mA)$.
The index of the equation chosen on iteration $k$ is $i_k\in[m]$. 
The matrix $\mA$ is presumed to have no rows of all zeros so that $\norm{\mAi}>0$ for all $i$ and the Kaczmarz update (\Cref{eqn:kacz-update}) is well defined for any $\ik$.

The initial iterate is denoted $\xinit$ and the iterate at iteration $k$ is denoted $\xk$. 
Likewise for SSRK methods, the initial selectable set is $\selectset_0$, often chosen as $\selectset_0 = [m]$, and subsequent selectable sets are denoted $\selectset_k$.
We define the complement of a selectable set as $\selectset^C=[m]\backslash\selectset$. 
In the analysis of the selectable set, we use the floor $\lfloor \ell \rfloor$ to denote the greatest integer less than or equal to $\ell$ and the ceiling $\lceil \ell \rceil$ to denote the smallest integer greater than or equal to $\ell$.
For ease of reference, in \Cref{tab:abrev} we list the acronyms for the methods that we investigate and analyze in this paper.

\begin{table}[ht]
\begin{center}
\begin{tabular}{lll} 
 \toprule
 Method acronym &  Method name & Reference(s) \\ \midrule
RK & Randomized Kaczmarz & Strohmer and Vershynin 2009 \cite{strohmer2009randomized} \\
MDK & Max-Distance Kaczmarz & Motzkin 1954 \cite{motzkin1954relaxation}\\
SSRK & Selectable Set Randomized Kaczmarz & \Cref{algo:selectablesetrk} \\
NSSRK & Non-Repetitive Selectable Set Randomized Kaczmarz & \Cref{algo:nssrk} \\
GSSRK & Gramian Selectable Set Randomized Kaczmarz & Nutini et al. 2016 \cite{nutini2016convergence} and \Cref{algo:gssrk} \\
GRK & Greedy Randomized Kaczmarz & Bai and Wu 2018 \cite{bai2018greedy} \\
RGRK & Relaxed Greedy Randomized Kaczmarz & Bai and Wu 2018 \cite{bai2018relaxed}\\
 \bottomrule
\end{tabular}
\caption{Acronyms of the methods discussed.}
\label{tab:abrev}
\end{center}
\end{table}

The scalars $p_1, p_2, \ldots, p_\nrows$ represent probabilities associated with each equation of the system or equivalently each row of $\mA$. We often refer to rows of $\mA$ and equations $\mAi \vx = \vbi$ interchangeably. We use $\Diag (\mathbf{v})$ to denote the square matrix whose diagonal entries take the values from the vector $\mathbf{v}$ and whose remaining entries are all $0$. In particular, we utilize the diagonal matrices of row norms $\mD = \Diag( \norm{\mA_1}, \norm{\mA_2}, \ldots, \norm{\mA_\nrows})$ and probabilities $\mP=\Diag(p_1, p_2, \ldots, p_\nrows)$.

The Gramian matrix of $\mA$ is $\mG=\mA \mA^T$ where $\mG \in \R^{m \times m}$. 
It has the property $G_{ij} = \langle \mA_i, \mA_j \rangle$ where $\langle \cdot, \cdot \rangle$ is the dot product between vectors. 
By the symmetry of the dot product, $\mG$ is symmetric with $G_{ij}=G_{ji}=0$ if and only if rows $\mA_i$ and $\mA_j$ are orthogonal. 
Thus, those entries with a nonzero value indicate that the corresponding rows of $\mA$ are non-orthogonal. 
From the Gramian matrix $\mG$, we derive a \emph{non-orthogonality graph} where each node represents a row of $\mA$ and a nonzero entry $\mG_{ij}$ is interpreted as an edge between nodes $i$ and $j$. 
We always assume that our graphs do not contain self edges and allow $G_{ii} \neq 0$ from here and thereafter.  
Since each node in the non-orthogonality graph represents a row in the matrix $\mA$, the number of nodes in the graph is $\nrows$. For this graph, we let $\indepset$ denote the maximum independent set, the largest set in which no pair of nodes share an edge. We denote the cardinality and complement of a set as $\abs{\indepset}$ and $\indepset^C=[m]\backslash\indepset$, respectively.




\section{Selectable set method}\label{sec:selectset}

Since the convergence behavior of the Kaczmarz method is highly dependent on the sampling strategy used to determine the order of projections, it is important to develop and analyze various sampling techniques. Here, we focus on the framework in which a \emph{selectable set} of equations is identified in each iteration and then an equation is selected from that set, typically at random.  Since the Kaczmarz update only improves the solution when selecting an equation that is not already solved, we aim to identify the selectable set that is precisely the set of equations not currently solved. 

\begin{definition} \label{def:selectset}
    A selectable set $\selectset_k\subset [m]$ for a Kaczmarz method, given a matrix $\mA$, vector $\vb$ and an iterate $\xk$, is a set of indices that satisfies $i \not \in \selectset_k \implies \mAi \xk = \vbi$.
\end{definition}

Based on this definition, if an equation $i$ is sampled from outside the selectable set, then $\mAi \xk = \vbi$, implying that if row $i$ were chosen for the Kaczmarz update, then $\xkpo = \xk$. Thus, sampling exclusively from the selectable set automatically guarantees faster convergence than that of a method that selects in the same random fashion from the entire set of equations\footnote{Note that this is related to, but fundamentally different from, random sampling without replacement. Sampling without replacement indeed guarantees that the same equation is not selected in consecutive iterations, but an equation solved in iteration $k$ need not be solved in even the next iteration. See \Cref{sec:future} for more discussion.}. 

In the Selectable Set Randomized Kaczmarz method (SSRK), the equation chosen at each iteration must be sampled from the current selectable set. 
We assume that a fixed probability distribution on the equations is given.
Then, instead of sampling $\ik$ according to the probabilities $p_1, p_2, \ldots, p_\nrows$ as in RK, SSRK samples $\ik$ conditioned on $\ik \in \selectset_k$.
This can be achieved by repeatedly sampling $\ik$ according to the probabilities $p_1, p_2, \ldots, p_\nrows$ until the condition $\ik \in \selectset_k$ is satisfied. \edit{This rejection sampling is mathematically equivalent to sampling from the explicit distribution $p_i / \sum_{j \in \selectset_k} p_j$ for $i \in \selectset_k$ and zero for $i \notin \selectset_k$ at each iteration. Explicitly updating the sampling distribution at each iteration is advantageous and yields computational improvement when the selectable set is small and explicitly known. Conversely, rejection sampling is advantageous if the selectable set contains a majority of the rows because it bypasses the computational overhead of recomputing the distribution.}

\begin{algorithm}[ht]
\begin{algorithmic}[1]
    \State{\textbf{Input}  Matrix $\mA$, RHS $\vb$, initial selectable set $\selectset_0$, initial iterate $\xinit\in\row(\mA)$, probabilities $p_1, p_2, \ldots, p_\nrows>0$}
    \For{$k=0, 1, \ldots$}
    \State Sample row $i_k$ according to probabilities $p_1, p_2, \ldots, p_\nrows$ with rejection until $i_k\in  \selectset_k$ \edit{\footnotemark}
    \State  Update $\xkpo = \xk -\frac{\mAik \xk - \vbik}{\normsq{\mAik}} \mA_{i_k}^T$
    \State Update $\selectset_{k+1}$ so that $i \not \in \selectset_{k+1} \implies \mAi \xkpo = \vbi$
    \Comment{See \Cref{algo:nssrk,algo:gssrk} for examples.}
    \EndFor
    \State{\textbf{Output} Approximate solution $\xk$}
    
\end{algorithmic}
\caption{Selectable Set Randomized Kaczmarz (SSRK)}\label{algo:selectablesetrk}
\end{algorithm}
\footnotetext{\edit{This rejection sampling strategy is equivalent to sampling from the distribution defined by $p_i / \sum_{j \in \selectset_k} p_j$ for $i \in \selectset_k$ and zero for $i \notin \selectset_k$, computing this distribution explicitly is more advantageous if the size of the selectable set is small.}}


\subsection{Non-repetitive selectable set}


A simple construction to update the selectable set $\selectset$ is to begin by including every index in the first selectable set $\selectset_0 = [\nrows]$.
Then, for each subsequent iteration, omit the most recently chosen index from the selectable set so that $\selectset_{k+1} = [m]\backslash \{\ik\}$. 
In this construction, $\abs{\selectset_0} = \nrows$ and $\abs{\selectset_k} = \nrows -1$ for $k > 0$. 
In order to save memory, there is no need to explicitly construct $\selectset_k$. 
It is sufficient to keep track of the previously sampled row which corresponds to $\selectset_k^C = \{i_{k-1}\}$.
Sampling with rejection from this selectable set has a probability $(\nrows-1)/\nrows$ of succeeding on each attempt, since $\abs{\selectset_k}=\nrows-1$.
The total number of attempts required to sample $\ik$ is thus geometrically distributed with mean $\nrows/(\nrows-1)$. We refer to this method as the Non-Repetitive Selectable Set method (NSSRK) \Cref{algo:nssrk}.

\begin{algorithm}[ht]
\begin{algorithmic}[1]
    \State{\textbf{Input}  Matrix $\mA$, RHS $\vb$, initial iterate $\xinit\in\row(\mA)$, probabilities $p_1, p_2, \ldots, p_\nrows>0$}
    \State $\selectset_0 = [\nrows ]$
    \For{$k=0, 1, \ldots$}
    \State Sample row $i_k$ according to probabilities $p_1, p_2, \ldots, p_\nrows$ with rejection until $i_k\in  \selectset_k$
    \State  Update $\xkpo = \xk -\frac{\mAik \xk - \vbik}{\normsq{\mAik}} \mAik^T$
    \State Set $\selectset_{k+1} = [m] \backslash \{i_k\}$
    \EndFor
    \State{\textbf{Output} Approximate solution $\xk$}
    
\end{algorithmic}
\caption{Non-Repetitive Selectable Set Randomized Kaczmarz (NSSRK)}\label{algo:nssrk}
\end{algorithm}

\FloatBarrier
\subsection{Gramian-Based Selectable Set}





A second method to update the selectable set, originally proposed in Nutini et al. is to leverage the Gramian $\mG = \mA \mA^T$ of the matrix $\mA$ \cite{nutini2016convergence, sepehry2016finding}. In many structured problems we have access to both the matrix $\mA$ and its Gramian $\mG$. 
One example of such a problem is graph semi-supervised learning \cite{bertozzi2018uncertainty}. 
The Gramian, by definition, has the property that $G_{ij} = \langle \mA_i, \mA_j \rangle$. That is, the $ij\thup$ entry of the Gramian is the inner product between the $i\thup$ and $j\thup$ rows of $\mA$.
So  $G_{ij} = 0$ if and only if rows $\mA_i$ and $\mA_j$ are orthogonal.
Based on \Cref{lem:Gramian-orthogonality}, stated and proven below, we will develop an update to the selectable set based on the Gramian. 

\begin{lemma} \label{lem:Gramian-orthogonality}
    If an equation $\mA_j \vx = b_j$ solved by the iterate $\xk$, and if $\mAik$ is orthogonal to $\mA_j$ (i.e. $G_{\ik j}=0$), then the equation is also solved by the next iterate $\xkpo$.
\end{lemma}

\begin{proof}
    Let $\xk$ and $j$ satisfy $\mA_j \xk = b_j$, and suppose $\mAik$ is orthogonal to $\mA_j$.
    Multiplying by $\mA_j$ on the left of both sides of the Kaczmarz update (\Cref{eqn:kacz-update}) results in
    \begin{align*}
        \mA_j \xkpo & = \mA_j \left(\xk -\frac{\mAik x^k - \vbik}{\normsq{\mAik}} \mAik^T \right)\\
        & = \mA_j \xk -\mA_j \frac{\mAik x^k - \vbik}{\normsq{\mAik}} \mAik^T \\
        & =\mA_j \xk - \frac{\mAik x^k - \vbik}{\normsq{\mAik}} \mA_j \mAik^T \\
        &= \mA_j \xk - \frac{\mAik x^k - \vbik}{\normsq{\mAik}}  \langle \mAik, \mA_j \rangle.
    \end{align*}
    Using the assumption that $\mA_j \xk = b_j$, 
    \begin{equation*}
        \mA_j \xkpo 
        = b_j - \frac{\mAik x^k - \vbik}{\normsq{\mAik}}  \langle \mAik, \mA_j \rangle.
    \end{equation*}
    Finally, by the assumption that $\mAik$ is orthogonal to $\mA_j$, 
    \begin{align*}
    \mA_j \xkpo  &=  \vbj -\frac{\mAik x^k - \vbik}{\normsq{\mAik}} 0 \\&= \vbj.
    \end{align*}
\end{proof}

Recall that any equation $\mA_j \vx = b_j$ that is not selectable must be solved by the iterate $\xk$. Thus, from \Cref{lem:Gramian-orthogonality}, we know that if $j \not \in \selectset_k$ and if $\ik$ satisfies $G_{\ik j} = 0$, then the equation $\mA_j \vx = \vbj$ is still solved by the next iterate, i.e. $\mA_j \xkpo = \vbj$. 
This suggests that any unselectable index $j \not \in \selectset_k$ for which $G_{\ik j} = 0$ should remain unselectable on iteration $k+1$, since the corresponding equation is still solved.
The Gramian Selectable Set Randomized Kaczmarz method (GSSRK), \Cref{algo:gssrk}, is based on this observation. In GSSRK, only those indexes $j$ with $G_{\ik j} \neq 0$ are reintroduced to the selectable set at each iteration.


\begin{algorithm}[ht]
\begin{algorithmic}[1]

    \State \textbf{Input}  Matrix $\mA$, RHS $\vb$, Gramian $\mG := \mA \mA^T$, initial iterate $\xinit\in\row(\mA)$, probabilities $p_1, p_2, \ldots, p_\nrows>0$
    \State{$\selectset_0 = [m]$}
    \For{$k=0, 1, \ldots$}
    \State Sample row $i_k$ according to probabilities $p_1, p_2, \ldots, p_\nrows$ with rejection until $i_k\in  \selectset_k$
    \State  Update $\xkpo = \xk -\frac{\mAik \xk - \vbik}{\normsq{\mAik}} \mAik^T$
    \State $\selectset_{k+1} = (\selectset_{k} \cup \{j: G_{i_kj} \neq 0 \}) \backslash \{i_k\}$
    \EndFor
    \State{\textbf{Output} Approximate solution $\xk$}
    
\end{algorithmic}
\caption{Gramian Selectable Set Randomized Kaczmarz (GSSRK) \cite{nutini2016convergence}}\label{algo:gssrk}
\end{algorithm}


\section{Convergence analysis}
\label{sec:convergence}

Now, we turn to proving convergence results for \Cref{algo:selectablesetrk}. 
First, we prove a one-step convergence result for the general selectable set method with a fixed probability distribution. 
We analyze a single iteration of the general case with an arbitrary sampling distribution, and a known selectable set $S_k$. 
We then focus on specific probability distributions common in the literature \cite{strohmer2009randomized}, and prove an improvement in the convergence constant which is inversely proportional to the size of the selectable set.  This speedup is roughly what one should expect.  For example if rows are sampled uniformly, then RK wastes an \edit{$1 - |S_k|/m$} fraction of iterations on updates which make no progress, whereas SSRK avoids this.

\begin{theorem} \label{thm:selectable-one-step}
The iterates of Selectable Set Randomized Kaczmarz (\Cref{algo:selectablesetrk}) satisfy
\begin{equation*}
    \Ek \normsq{\xkpo - \xopt} \le  \left (1 - \frac{\sigma^2_{\min}(\mathbf{P}^\frac{1}{2}\mathbf{D}^{-1}\mA)}{\sum_{j \in \selectset_{k}} p_j } \right ) \normsq{\xk - \xopt},
\end{equation*}
where $\xopt$ is the least-norm solution, $\mathbf{P} = \Diag(p_1, p_2, \ldots, p_\nrows)$, and $\mathbf{D} = \Diag( \norm{\mA_1}, \norm{\mA_2}, \ldots, \norm{\mA_\nrows})$ when $\sum_{j \in \selectset_{k}} p_j  \neq 0$.
\end{theorem}

\begin{proof}
From the update formula \Cref{eqn:kacz-update}, we derive the usual update for the squared error
\begin{equation*}
    \normsq{\xkpo - \xopt} = \normsq{\xk - \xopt} - \frac{\abssq{\mAik \xk - \vbik}}{\normsq{\mAik}}.
\end{equation*}
Letting $\Ek$ denote the expectation conditioned on $i_0, i_1, \ldots, i_{k-1}$, we take this conditional expectation on both sides
\begin{align*}
    \Ekb{\normsq{\xkpo - \xopt}} 
    &= \normsq{\xk - \xopt} - \Ekb{\frac{\abssq{\mAik \xk - \vbik}}{\normsq{\mAik}}} \\
    &= \normsq{\xk - \xopt} - \sum_{i \in \selectset_{k}} \frac{p_i}{\sum_{j \in \selectset_k} p_j} \frac{\abssq{\mAi \xk - \vbi}}{\normsq{\mAi}}.
\end{align*}
Pulling the normalizing constant $1/ \sum_{j \in \selectset_k} p_j$ out of the summation,
\begin{align*}
    \Ekb{\normsq{\xkpo - \xopt}} 
    &= \normsq{\xk - \xopt} - \frac{1}{\sum_{j \in \selectset_k} p_j}\sum_{i \in \selectset_{k}} p_i \frac{\abssq{\mAi \xk - \vbi}}{\normsq{\mAi}}.\\
\end{align*}


By the definition of the selectable set, we know that $i \not\in \selectset_{k} \implies \mAi \xk - \vbi = 0$ so we can extend our sum over all rows

\begin{align*}
    \Ekb{\normsq{\xkpo - \xopt}} 
    &= \normsq{\xk - \xopt} - \frac{1}{\sum_{j \in \selectset_k} p_j}\sum_{i \in [m]} p_i \frac{\abssq{\mAi \xk - \vbi}}{\normsq{\mAi}}.
\end{align*}
By definition of $\xopt$ we  can rewrite $\vbi = \mAi \xopt$

\begin{align*}
    \Ekb{\normsq{\xkpo - \xopt}} 
    &= \normsq{\xk - \xopt} - \frac{1}{\sum_{j \in \selectset_k} p_j}\sum_{i \in [m]} p_i \frac{\abssq{\mAi \xk - \mAi \xopt}}{\normsq{\mAi}} \\
    &= \normsq{\xk - \xopt} - \frac{1}{\sum_{j \in \selectset_k} p_j}\sum_{i \in [m]} p_i \frac{\abssq{\mAi( \xk - \xopt)}}{\normsq{\mAi}}\\
    &= \normsq{\xk - \xopt} - \frac{1}{\sum_{j \in \selectset_k} p_j}\sum_{i \in [m]} \abssq{\frac{p_i^{\frac{1}{2}}}{\norm{\mAi}} \mAi( \xk - \xopt)}.
\end{align*}

We now rewrite the summation as a norm of a matrix-vector multiplication using the previously defined $\mathbf{P}$ and $\mathbf{D}$ matrices: 

\begin{align*}
    \Ekb{\normsq{\xkpo - \xopt}} 
    &= \normsq{\xk - \xopt} - \frac{1}{\sum_{j \in \selectset_k} p_j} \normsq{ \mathbf{P}^{\frac{1}{2}} \mathbf{D}^{-1} \mA (\xk - \xopt)}.
\end{align*}




Now, we establish the lower bound $\normsq{ \mathbf{P}^{\frac{1}{2}} \mathbf{D}^{-1} \mA (\xk - \xopt)} \geq \sigma^2_{\min}(\mathbf{P}^{\frac{1}{2}} \mathbf{D}^{-1} \mA) \normsq{\xk - \xopt}$, where $\sigma_{\min}$ is the smallest \emph{nonzero} singular value based on the proof technique in Zouzias and Freris \cite{zouzias2013randomized}.
By the assumption that $\mA$ has no zero rows, $\mD$ is symmetric positive definite (SPD). Likewise, since the probabilities $p_1, p_2, \ldots, p_\nrows$ are positive, $\mP$ is SPD.
Since $\mD$ and $\mP$ are SPD, so are $\mathbf{P}^{\frac{1}{2}}$ and $\mathbf{D}^{-1}$.
Thus, $\row(\mathbf{P}^{\frac{1}{2}} \mathbf{D}^{-1} \mA)$ is equal to $\row(\mA)$. The vector $\xopt$ belongs to $\row(\mA)$ because it is the least-norm solution to $\mA\vx=\vb$. Additionally, the iterate $\xk$ belongs to $\row(\mA)$ because the initial iterate $\xinit$ belongs to $\row(\mA)$ as does the direction of the Kaczmarz update at each iteration. Since the iterate $\xk$ and the least-norm solution $\xopt$ both belong to $\row(\mA)$, so does their difference $\xk-\xopt$. Finally, recalling that $\row(\mA)=\row(\mathbf{P}^{\frac{1}{2}} \mathbf{D}^{-1} \mA)$, we see that $\xk-\xopt \in \row(\mathbf{P}^{\frac{1}{2}} \mathbf{D}^{-1} \mA)$ and thus $\normsq{ \mathbf{P}^{\frac{1}{2}} \mathbf{D}^{-1} \mA (\xk - \xopt)} \geq \sigma^2_{\min}(\mathbf{P}^{\frac{1}{2}} \mathbf{D}^{-1} \mA) \normsq{\xk - \xopt}$. Applying this lower bound, we arrive at the desired result
\begin{align*}
    \Ekb{\normsq{\xkpo - \xopt}} 
    &= \normsq{\xk - \xopt} - \frac{1}{\sum_{j \in \selectset_k} p_j} \normsq{ \mathbf{P}^{\frac{1}{2}} \mathbf{D}^{-1} \mA (\xk - \xopt)} \\
    & \leq \normsq{\xk - \xopt} - \frac{1}{\sum_{j \in \selectset_k} p_j} \sigma^2_{\min}(\mathbf{P}^{\frac{1}{2}} \mathbf{D}^{-1} \mA) \normsq{ \xk - \xopt} \\
    & =  \left (1 - \frac{\sigma^2_{\min}(\mathbf{P}^{\frac{1}{2}} \mathbf{D}^{-1} \mA)}{\sum_{j \in \selectset_{k}} p_j } \right ) \normsq{\xk - \xopt}. 
\end{align*}

\end{proof} 

\subsection{Corollaries}


We now take a closer at how \Cref{thm:selectable-one-step} applies to \Cref{algo:selectablesetrk,algo:gssrk,algo:nssrk} for two specific common choices of the probabilities $p_1, p_2, \ldots, p_\nrows$. 
In particular, we analyze uniform probabilities $p_i=1/m$ in \Cref{cor:gramian-uniform} and squared row norm probabilities  $p_i=\normsq{\mAi}/\normsq{\mA}_F$ in \Cref{cor:gramian-rownorm}.
For either choice of probabilities, the convergence guarantees vary depending on the selectable set $\selectset_k$ at each iteration.

\begin{corollary} 
\label{cor:gramian-uniform}
    When using uniform probabilities $p_i = 1/ \nrows$, the iterates of \Cref{algo:selectablesetrk} satisfy
    \begin{equation*}
        \Ek \normsq{\xkpo - \xopt} \le  \left (1 - \frac{\sigma^2_{\min}(\mD^{-1} \mA)}{\abs{ \selectset_{k}}} \right ) \normsq{\xk - \xopt},
    \end{equation*}
    where $\mathbf{D} = \Diag( \norm{A_1}, \norm{A_2}, \ldots, \norm{A_m})$.
\end{corollary}
\begin{proof}
Apply \Cref{thm:selectable-one-step} with $p_i=\frac{1}{m}$.

\end{proof}

\Cref{cor:gramian-uniform} shows that \Cref{algo:selectablesetrk} with uniform probabilities $p_i = 1/ \nrows$ acheives a convergence guarantee that depends on the number of selectable rows $\abs{\selectset_k}$. 
The fewer the number of selectable rows, the faster the convergence guarantee.
If nearly all rows are selectable at every iteration, \Cref{cor:gramian-uniform} recovers the known convergence guarantee $\left (1 - \sigma^2_{\min}(\mD^{-1} \mA)/\nrows \right )$ for RK with uniform probabilities \cite{nutini2016convergence}. 
In \Cref{algo:nssrk}, all but one row are selectable at each iteration, so
the convergence guarantee is trivially slightly faster than that of RK as shown in \Cref{cor:n-uniform}.

\begin{corollary}\label{cor:n-uniform}
    When using uniform probabilities $p_i = \frac{1}{\nrows}$, the iterates of \Cref{algo:nssrk} satisfy
    \begin{align*}
        \Einit \normsq{\xone - \xopt} &\le \left (1 - \frac{\sigma^2_{\min}(\mD^{-1} \mA)}{\nrows} \right ) \normsq{\xinit - \xopt} \\
        \text{and} \qquad \Ek \normsq{\xkpo - \xopt} &\le \left (1 - \frac{\sigma^2_{\min}(\mD^{-1} \mA)}{\nrows - 1} \right ) \normsq{\xk - \xopt} \qqtext{for} k \ge 1,
    \end{align*}
    where $\mathbf{D} = \Diag( \norm{A_1}, \norm{A_2}, \ldots, \norm{A_\nrows})$. 
\end{corollary}
\begin{proof}
    Substitute $\abs{\selectset_0}=m$ and $\abs{\selectset_k}=m-1$ for $k\ge 1$ in \Cref{cor:gramian-uniform}.
\end{proof}

As discussed, using uniform probabilities results in a simple relationship between the number of selectable rows and the convergence guarantee at each iteration.
In contrast, using squared row norm probabilities $p_i = \normsq{\mA_i}/\normsq{\mA}_F$ results in a slightly more complicated relationship between the selectable set $\selectset_k$ and the convergence guarantee.
This relationship is shown in \Cref{cor:gramian-rownorm}.

\begin{corollary} \label{cor:gramian-rownorm}
    When using probabilities proportional to the squared row norms $p_i = \normsq{\mA_i} / \normsq{\mA}_F$, the iterates of \Cref{algo:selectablesetrk} satisfy
    \begin{equation*}
    \Ek \normsq{\xkpo - \xopt} \le  \left (1 - \frac{\sigma^2_{\min}(\mA)}{\sum_{j \in \selectset_k } \normsq{\mA_j} } \right ) \normsq{\xk - \xopt}.
    \end{equation*} 
\end{corollary}

\begin{proof}

    When $p_i = \normsq{\mA_i} / \normsq{\mA}_F$, we have
    \begin{equation*}
        \mP^\frac{1}{2}\mD^{-1} 
        = \Diag\parens{
        \frac{\norm{\mA_1}}{\norm{\mA}_F},
        \frac{\norm{\mA_2}}{\norm{\mA}_F}, 
        \ldots, 
        \frac{\norm{\mA_\nrows}}{\norm{\mA}_F}} \Diag\parens{\frac{1}{\norm{\mA_1}},\frac{1}{\norm{\mA_2}}, \ldots, \frac{1}{\norm{\mA_\nrows}}} = \frac{1}{\norm{\mA}_F} \mI.
    \end{equation*}
    Substituting this into \Cref{thm:selectable-one-step} along with the probabilities $p_i = \normsq{\mA_i} / \normsq{\mA}_F$, we find
    \begin{align*}
        \Ek \normsq{\xkpo - \xopt} & \le \left (1 - \frac{\sigma^2_{\min}(\frac{1}{\norm{\mA}_F}\mA)}{\sum_{j \in \selectset_{k}} \frac{\normsq{\mA_j}}{\normsq{\mA}_F}} \right ) \normsq{\xk - \xopt} \\
        & =  \left (1 -  \frac{\frac{1}{\normsq{\mA}_F}\sigma^2_{\min}(\mA)}{\frac{1}{\normsq{\mA}_F} \sum_{j \in \selectset_{k}} \normsq{\mA_j}} \right ) \normsq{\xk - \xopt}\\
        & = \left (1 - \frac{\sigma^2_{\min}(\mA)}{\sum_{j \in \selectset_k} \normsq{\mA_j} } \right ) \normsq{\xk - \xopt}.
    \end{align*}
\end{proof}

\Cref{cor:gramian-rownorm} shows that \Cref{algo:selectablesetrk} with squared row norm probabilities $p_i = \normsq{\mA_i} / \normsq{\mA}_F$ acheives a congergence guarantee that depends on the quantity $\sum_{j \in \selectset_k} \normsq{\mA_j}$.
This quantity is the squared Fr\"obenius norm of the row-submatrix of $\mA$ composed of those rows that are selectable.
When all of the rows of $\mA$ have roughly the same norm, \Cref{cor:gramian-rownorm} suggests that \Cref{algo:selectablesetrk} converges faster when fewer rows are selectable.
When the rows have very different norms, the relationship between the selectable set and the convergence guarantee is not so simple.

When nearly all rows are selectable at every iteration, 
\Cref{cor:gramian-rownorm} recovers the known convergence guarantee $\left (1 - \sigma^2_{\min}(\mA) / \normsq{\mA}_F \right )$ for RK with squared row norm probabilities $p_i = \normsq{\mA_i} / \normsq{\mA}_F$ \cite{strohmer2009randomized}.
In \Cref{algo:nssrk}, all but one row are selectable at each iteration, so the convergence guarantee is expectedly similar to that of RK as shown in \Cref{cor:n-rownorm}.

\begin{corollary} \label{cor:n-rownorm}
    When using probabilities proportional to the squared row norms $p_i = \normsq{\mA_i} / \normsq{\mA}_F$, the iterates of \Cref{algo:nssrk} satisfy
    \begin{align*}
        \Einit \normsq{\xone - \xopt} 
        &\le \left (1 - \frac{\sigma^2_{\min}(\mA)}{\normsq{\mA}_F }\right ) \normsq{\xinit - \xopt} \\
        \text{and} \qquad \Ek \normsq{\xkpo - \xopt} 
        &\le \left (1 - \frac{\sigma^2_{\min}(\mA)}{\normsq{\mA}_F - \normsq{\mA_{i_{k-1}}}}\right ) \normsq{\xk - \xopt} \qqtext{for} k \ge 1.
    \end{align*}
\end{corollary}
\begin{proof}
    For iteration $k=0$, substitute $\selectset_0=[m]$ in \Cref{cor:gramian-rownorm}.
    For iterations $k\ge 1$, substituting $\selectset_{k} = [m]\backslash \{i_{k-1}\}$ in \Cref{cor:gramian-rownorm} shows the desired result
    \begin{align*}
        \Ek \normsq{\xkpo - \xopt}  & \le \left (1 - \frac{\sigma^2_{\min}(\mA)}{\sum_{j \in [m]\backslash \{i_{k-1}\}} \normsq{\mA_j} } \right ) \normsq{\xk - \xopt} \\
        & = \left (1 - \frac{\sigma^2_{\min}(\mA)}{\sum_{j \in [m]} \normsq{\mA_j} - \normsq{\mA_{i_{k-1}}}} \right ) \normsq{\xk - \xopt} \\
        & = \left (1 - \frac{\sigma^2_{\min}(\mA)}{\normsq{\mA}_F - \normsq{\mA_{i_{k-1}}}} \right ) \normsq{\xk - \xopt}.
    \end{align*}
\end{proof}


From \Cref{cor:n-rownorm}, we see that NSSRK (\Cref{algo:nssrk}) is expected to converge faster on iterations when some row with large norm is not selectable.  This is essentially an artifact of the sampling scheme. Rows with large norms are chosen disproportionately often; when such a row is not selectable, the sampling is more uniform, and convergence improves. 

Next we compare the convergence analysis of NSSRK with convergence guarantee from the popular Kaczmarz method proposed in Bai and Wu \cite{bai2018relaxed}. While the NSSRK method is relatively simple and does not yield large improvement over non-selectable set methods, the convergence guarantee corresponds to the same convergence guarantee as in \cite{bai2018relaxed}. This leads us to believe that there may be a theory gap in Kaczmarz convergence analysis.

\subsection{Comparison with Relaxed Greedy Randomized Kaczmarz (RGRK) theory} \label{sec:bai-wu-compare}

The Relaxed Greedy Randomized Kaczmarz method (RGRK)\cite{bai2018relaxed} is similar to MDK in that both methods use sampling strategies that are biased toward rows with larger normalized residuals \edit{$\abs{\mAi\xk-\vbi}/\norm{\mAi}$}. 
In particular, RGRK considers only rows that satisfy
\begin{equation}\label{eqn:rgrk-thresh}
    \frac{\abssq{\mAi\xk-\vbi}}{\normsq{\mAi}} 
    \ge 
    \theta\max_{j \in [\nrows]}\parens{\frac{\edit{\abssq{\mA_j \xk - \vbj}}}{\normsq{\mA_j}}} +\parens{1-\theta}\frac{\normsq{\mA \xk - \vb}}{\normsq{\mA}_F}
    \qquad
    \text{for some }
    \theta\in[0,1]
\end{equation}
and samples the row $\mAik$ with probability proportional to the squared residual $\abssq{\mAi\xk-\vbi}$ from among such rows.
RGRK satisfies the convergence result\cite{bai2018relaxed}
\begin{equation}\label{eqn:bai-wu-rate}
    \Ek \normsq{\xkpo - \xopt} \le  \parens{1 -  \parens{\theta\frac{\normsq{\mA}_F}{\gamma}  + (1-\theta)} \frac{\sigma^2_{\min}(\mA)}{\normsq{\mA}_F}} \normsq{\xk - \xopt}
    \qqtext{for} k \ge 1,
\end{equation}
where 
\begin{equation}
\label{const:gamma}
    \gamma = \max_{i \in [\nrows]} \sum_{j = 1, j \neq i }^{\nrows} \normsq{\mA_j} = \normsq{\mA}_F - \min_{i \in [\nrows]}\normsq{\mAi}.
\end{equation}
This convergence result is optimized by the parameter $\theta = 1$ for which RGRK is equivalent to MDK. With $\theta=1$, \Cref{eqn:bai-wu-rate} simplifies to
\begin{equation}\label{eqn:bai-wu-best-case}
    \Ek \normsq{\xkpo - \xopt} \le \left (1 -  \frac{\sigma^2_{\min}(\mA)}{\gamma} \right ) \normsq{\xk - \xopt}\qqtext{for} k \ge 1.
\end{equation}
Coincidentally, NSSRK satisfies \Cref{eqn:bai-wu-best-case} for squared row norm probabilities.

\begin{corollary}
    When using probabilities proportional to the squared row norms $p_i = \normsq{\mA_i}/\normsq{\mA}_F$, the iterates of NSSRK (\Cref{algo:nssrk}) satisfy \Cref{eqn:bai-wu-best-case}.
\end{corollary}
\begin{proof}
    By \Cref{cor:n-rownorm}, we have the desired result
    \begin{align*}
        \Ek \normsq{\xkpo - \xopt} 
        &\le 
        \left (1 - \frac{\sigma^2_{\min}(\mA)}{\normsq{\mA}_F - \normsq{\mA_{i_{k-1}}}}\right ) \normsq{\xk - \xopt}
        \\ &\le 
        \left (1 - \frac{\sigma^2_{\min}(\mA)}{\normsq{\mA}_F - \min_{i \in [\nrows]}\normsq{\mAi}}\right ) \normsq{\xk - \xopt}
        \\ &=
        \left (1 - \frac{\sigma^2_{\min}(\mA)}{\gamma}\right ) \normsq{\xk - \xopt}.
    \end{align*}
\end{proof}

Since NSSRK satisfies \Cref{eqn:bai-wu-best-case} while RGRK only satisfies \Cref{eqn:bai-wu-rate}, one might incorrectly assume that NSSRK will outperform RGRK.
However, as we observe in \Cref{sec:experiments}, this is not the case.
RGRK significantly outperforms NSSRK and GSSRK, neglecting CPU time.
This discrepancy between convergence results and observed performance suggests that \Cref{eqn:bai-wu-rate} is not the tightest possible convergence result for RGRK.
Indeed, RGRK was recently shown to satisfy the tighter convergence result \cite{gower2021adaptive} \edit{based on the constant $\sigma_\infty^2(\mA)$, defined as }
\begin{equation}
\label{const:siginf}
\edit{
    \sigma_\infty^2(\mA) = \min_{\substack{\vx\in\row(\mA)\setminus\{\xopt\} \\ \exists j  \text{ s.t. } \mA_j(\vx - \xopt) = 0}}\parens{\max_i \frac{\abssq{\mAi(\vx-\xopt)}}{\normsq{\mAi}\normsq{\vx-\xopt}}},}
\end{equation}
\edit{where $\xopt \in \row(\mA)$. The improved convergence result is as follows: }

\begin{equation*}
    \Ek \normsq{\xkpo - \xopt} \le  \parens{1 - \theta \sigma_\infty^2(\mA) - (1-\theta)\frac{\sigma^2_{\min}(\mA)}{\normsq{\mA}_F}} \normsq{\xk - \xopt}
    \qqtext{for} \edit{k > 1}.
\end{equation*}
\edit{This convergence result is optimized by the parameter $\theta = 1$, see \Cref{lem:rgrkineq}, leading to the improved convergence bound of }

\begin{equation*}
    \edit{\Ek \normsq{\xkpo - \xopt} \le  \parens{1 - \sigma_\infty^2(\mA)} \normsq{\xk - \xopt}
    \qqtext{for} k > 1.}
\end{equation*}

\edit{The improvement of this new bound is explained by the relationship between $\sigma_\infty^2(\mA)$ and $\sigma_{\min}^2(\mA)$, which is matrix dependent and discussed in more detail in other works \cite{nutini2016convergence,gower2021adaptive}. The following lemma, \Cref{lem:rgrkineq}, attempts to explain this relationship (second $\leq$) and the fact that $\theta = 1$ is the optimal parameter (first $\leq$).}

\begin{lemma}
\label{lem:rgrkineq}
\edit{Let $\mA$ be a matrix then 
    \begin{equation}
   \edit{  
    \frac{\sigma^2_{\min}(\mA)}{\normsq{\mA}_F}
   \le 
   \frac{\sigma^2_{\min}(\mA)}{\gamma} 
   \le
   \sigma_\infty^2(\mA)}.
\end{equation}
Where $\sigma^2_{\min}(\mA)$ is the smallest \emph{nonzero} singular value, $ \sigma_\infty^2(\mA)$ is defined by \Cref{const:siginf} and $\gamma$ is defined by \Cref{const:gamma}.
}\end{lemma}
\begin{proof}
{
\edit{By definition $\gamma = \normsq{\mA}_F - \min_{i \in [\nrows]}\normsq{\mAi} \leq \normsq{\mA}_F$ so }
\begin{align*}
\edit{
\frac{\sigma^2_{\min}(\mA)}{\normsq{\mA}_F}
   \le 
   \frac{\sigma^2_{\min}(\mA)}{\gamma}. }
\end{align*}
\edit{Since $\sigma^2_{\min}(\mA)$ is the smallest nonzero singular value and $\xopt \in \row(\mA)$ we have the following: }
\begin{align*}
\sigma^2_{\min}(\mA) 
&= \min_{\vx \in \row(\mA) \setminus \{\xopt\}} \sum_i \frac{\abssq{\mA_i(\vx-\xopt)}}{\normsq{\vx-\xopt}}.  
\end{align*}
\edit{Multiplying and dividing by $p_i = \frac{\normsq{\mA_i}}{\normsq{\mA}_F}$ then rearranging terms yields}
\begin{align*}
\sigma^2_{\min}(\mA) 
&= \min_{\vx \in \row(\mA) \setminus \{\xopt\}} \sum_i \frac{\normsq{\mA_i}}{\normsq{\mA}_F} \frac{\normsq{\mA}_F}{\normsq{\mA_i}} \frac{\abssq{\mA_i(\vx-\xopt)}}{\normsq{\vx-\xopt}}  
\\&= \normsq{\mA}_F \min_{\vx \in \row(\mA) \setminus \{\xopt\}} \sum_i p_i  \frac{\abssq{\mA_i(\vx-\xopt)}}{\normsq{\mA_i} \normsq{\vx-\xopt}}.  
\end{align*}
\edit{Imposing the constraint that $\vx \in \row(\mA) \setminus \{\xopt\}$ \emph{and} $\exists j \text{ s.t. } \mA_j(\vx - \xopt) = 0$ then noting that $\abssq{\mA_j(\vx-\xopt)} = 0$ results in }
\begin{align*}
\sigma^2_{\min}(\mA) 
&= \normsq{\mA}_F \min_{\vx \in \row(\mA) \setminus \{\xopt\}} \sum_i p_i  \frac{\abssq{\mA_i(\vx-\xopt)}}{\normsq{\mA_i} \normsq{\vx-\xopt}}
\\&\le \normsq{\mA}_F  \min_{\substack{\vx\in\row(\mA)\setminus\{\xopt\} \\ \exists j  \text{ s.t. } \mA_j(\vx - \xopt) = 0}} \sum_i p_i  \frac{\abssq{\mA_i(\vx-\xopt)}}{\normsq{\mA_i} \normsq{\vx-\xopt}}
\\ &= \normsq{\mA}_F  \min_{\substack{\vx\in\row(\mA)\setminus\{\xopt\} \\ \exists j  \text{ s.t. } \mA_j(\vx - \xopt) = 0}} \sum_{i \neq j} p_i  \frac{\abssq{\mA_i(\vx-\xopt)}}{\normsq{\mA_i} \normsq{\vx-\xopt}}.
\end{align*}
\edit{Taking a max over the rows $i$ for the fraction of the summation }
\begin{align*}
\sigma^2_{\min}(\mA) 
&\le \normsq{\mA}_F  \min_{\substack{\vx\in\row(\mA)\setminus\{\xopt\} \\ \exists j  \text{ s.t. } \mA_j(\vx - \xopt) = 0}} \sum_{i \neq j} p_i  \frac{\abssq{\mA_i(\vx-\xopt)}}{\normsq{\mA_i} \normsq{\vx-\xopt}}
\\ &\le \normsq{\mA}_F  \min_{\substack{\vx\in\row(\mA)\setminus\{\xopt\} \\ \exists j  \text{ s.t. } \mA_j(\vx - \xopt) = 0}}  \max_{\ell} \frac{\abssq{\mA_{\ell}(\vx-\xopt)}}{\normsq{\mA_{\ell}} \normsq{\vx-\xopt}}  \sum_{i \neq j} p_i.
\end{align*}
\edit{Bounding $\sum_{i \neq j} p_i$ from above by $1 - \min_i p_i$, simplifying and applying the definition of $\sigma_\infty^2(\mA)$ }

\begin{align*}
\sigma^2_{\min}(\mA) 
&\le \normsq{\mA}_F  \min_{\substack{\vx\in\row(\mA)\setminus\{\xopt\} \\ \exists j  \text{ s.t. } \mA_j(\vx - \xopt) = 0}}  \max_{\ell} \frac{\abssq{\mA_{\ell}(\vx-\xopt)}}{\normsq{\mA_{\ell}} \normsq{\vx-\xopt}}  \sum_{i \neq j} p_i
\\ & \le \normsq{\mA}_F  \min_{\substack{\vx\in\row(\mA)\setminus\{\xopt\} \\ \exists j  \text{ s.t. } \mA_j(\vx - \xopt) = 0}}  \max_{\ell} \frac{\abssq{\mA_{\ell}(\vx-\xopt)}}{\normsq{\mA_{\ell}} \normsq{\vx-\xopt}}(1 - \min_i p_i)
\\ & = \normsq{\mA}_F  \sigma_\infty^2(\mA) (1 - \min_i p_i)
\end{align*}
\edit{Finally, distributing the $\normsq{\mA}_F$ and substituting the definition of $p_i$ and $\gamma$ }
\begin{align*}
\sigma^2_{\min}(\mA) 
&\le \normsq{\mA}_F  \sigma_\infty^2(\mA) (1 - \min_i p_i)
\\&= \sigma_\infty^2(\mA)(\normsq{\mA}_F - \min_i \normsq{\mA_i}) 
\\&= \sigma_\infty^2(\mA)\gamma.
\end{align*}
\edit{Thus, $ \frac{\sigma^2_{\min}(\mA)}{\gamma} \le \sigma_\infty^2(\mA)$ as desired.}
}
\end{proof}

\edit{The improvement of this new bound is explained by the difference between $\frac{\sigma^2_{\min}(\mA)}{\gamma}$ and $\sigma_\infty^2(\mA)$, which we quantify by the above \Cref{lem:rgrkineq}. It is matrix dependent and described in more detail in Gower et al. \cite{gower2021adaptive} and Nutini et al. \cite{nutini2016convergence}} 

\section{Lower bounds on size of Gramian selectable set size}\label{sec:lower-bd-gssrk}
Investigating the size of the selectable set is essential for understanding the convergence of the GSSRK method, \Cref{algo:gssrk} \cite{nutini2016convergence, sepehry2016finding}.  
In \Cref{sec:convergence}, we proved that the convergence of the GSSRK method is dependent on the size of the selectable set, the smaller the selectable set the better the convergence guarantee is. In this section, we prove a lower bound on the size of the Gramian selectable set after $O(m)$ iterations where $m$ is the number of rows in the matrix.
For many structured problems the lower bound on the size of the selectable set is $cm$ where $c \in (0,1)$. 
This means that the convergence guarantee only improves by a constant factor in comparison to Kaczmarz methods that do not use a selectable set.

In particular, here we prove a general method for calculating a lower bound on the Gramian based selectable set. Then we apply this lower bound to some structured problems. Our lower bound can also be used as a  heuristic to assess in which cases applying the GSSRK method could lead to a large speedup. When the Gramian is sparse, meaning that there is a lot of orthogonality between rows of the matrix, the GSSRK method will likely yield improved convergence rates.

To develop our lower bound on the selectable set, we consider the Gramian of the matrix, $\mG = \mA \mA^T$ as an adjacency matrix and examine the graph formed from this matrix, the \emph{non-orthogonality graph}. This graph contains $\nrows$ nodes, each node representing a row of the original matrix $\mA$, node $i$ represents row $\mA_i$. There is an edge between nodes $i$ and $j$ if $G_{ij} = \langle \mA_i, \mA_j \rangle \neq 0$, meaning that rows $i$ and $j$ of $\mA$ are not orthogonal. So the graph encodes the orthogonality structure of the original matrix $\mA$. Two rows are orthogonal if and only if they do not share an edge in the graph. Next we will use the non-orthogonality graph and a graph theoretic approach to develop a lower bound on the selectable set \edit{for \Cref{algo:gssrk}}.

\begin{figure}[ht]
\centering
\label{fig:GramianGraph}
\includegraphics[width=.5\linewidth]{ 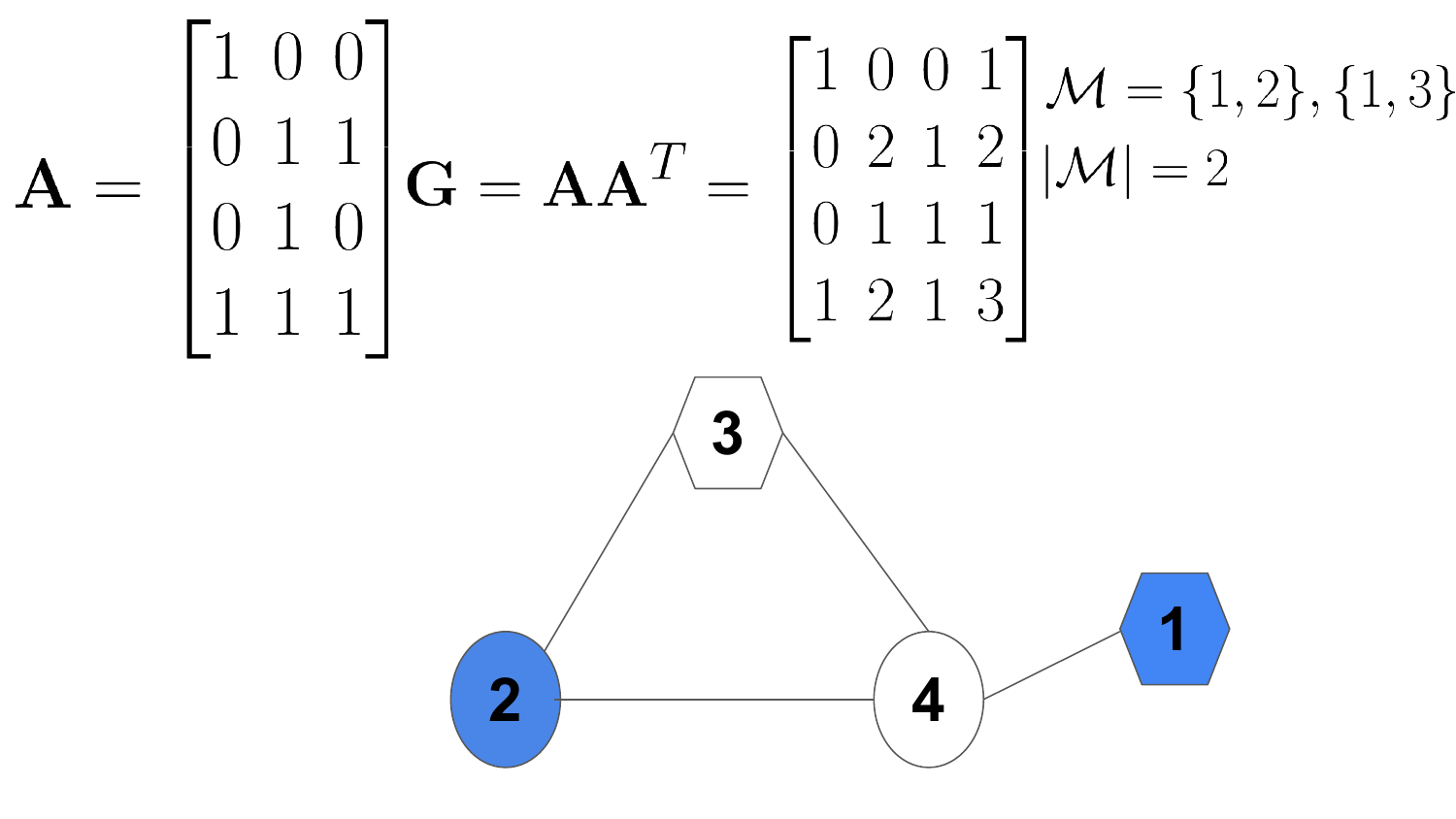}
\caption{\edit{The non-orthogonality graph constructed from the Gramian $\mG$ of the above matrix $\mA$. This undirected graph is connected because each node is reachable from every other node along the edges of the graph. The size of the maximal independent set of the graph is $2= \abs{\indepset}$ because the largest set of nodes that that do not share an edge among them is 2. If we consider the subgraph formed from nodes 1 and 2 or nodes 1 and 3, these graph have no edges. No larger sets can be created because if we add any additional nodes to these sets, the induced subgraphs will contain edges.}}
\end{figure}
\newpage
\begin{theorem}\label{thm:selectable-bound}
\edit{Given a system of equations $\mA, \vb$ and an initial iterate $\xinit$ such that $\mA \xinit \neq \vb$. Then the selectable set $\selectset_k$ from \Cref{algo:gssrk} satisfies $\nrows - \abs{\indepset} \leq \abs{\selectset_{k}} \leq \nrows-1 $ after $O(\nrows)$ iterations for all $k$. Where $\abs{\indepset}$ is the size of the maximal independent set formed by considering the Gramian matrix $\mG = \mA \mA^T$ as an adjacency matrix of a \emph{connected graph}\footnote{\edit{Connected graph corresponds to the Gramian matrix G not being a permutation of a block diagonal matrix, this is an assumption on the matrix $\mA$ not having a row or set of rows that are mutually orthogonal to all other rows.}}}.
\end{theorem}
\begin{proof}
    Consider the graph constructed from the Gramian of our matrix. Since this graph is connected and we have an unsolved equation, after $O(m)$ iterations each equation will become unsolved at least once and therefore be selectable. We now consider the maximum size of the set of unselectable rows after $O(m)$ iterations. The only rows that are unselectable are rows that have been selected by \edit{GSSRK}, meaning that all neighbors of this node must be selectable by the sampling process. Therefore if $\selectset_{k}^{c}$ is the set of unselectable rows. If $r_1, r_2 \in \selectset_{k}^{c}$, then $r_1$ and $r_2$ cannot share an edge. Otherwise, the row more recently selected in the method, without loss of generality $r_1$ is necessarily unselectable then $r_2$ would necessarily become selectable because it is a neighbor of $r_1$. Thus, each row in $\selectset_{k}^{c}$ must not be a neighbor of any other row in $\selectset_{k}^{c}$. So the maximum size of $\selectset_{k}^{c}$ is at most the size of the maximum independent set of the graph, $\indepset$. Thus the size of the selectable set is at least: $\nrows-\abs{\indepset}$.  \\
    For $k \geq 1$, $\selectset_{k}$ does not include row $i_{k-1}$ by the definition of the Kaczmarz update, at least one row is always solved, so $\abs{\selectset_{k}} \leq \nrows - 1$. 
\end{proof}

An algebraic interpretation of $\abs{\indepset}$ is that this number corresponds to the size of the largest set of rows of $\mA$ such that all elements in this set are pairwise-orthogonal. 
Equivalently, this is the number of rows in the largest submatrix of the rows of $\mA$ that can be constructed $\mQ \in \R^{\abs{\indepset}\times n}$, such that $\mQ \mQ^T = \mD$ where $\mD$ is a diagonal matrix. Since at most $\mA$ has $\abs{\indepset}$ pairwise orthogonal rows, by \Cref{lem:Gramian-orthogonality} all of the equations $\mAi \vx = b_i$ where $\mAi$ is a row of $\mQ$ can be simultaneously solved by a Kaczmarz method.

Next, we prove that the proposed lower bound of the the selectable set in \Cref{thm:selectable-bound} is tight and describe a row sampling scheme to achieve this lower bound.

\begin{proposition}\label{prop:tightbf}
There exists a sampling strategy for a Gramian based Selectable Set Kaczmarz method that achieves the lower bound of $\nrows-\abs{\indepset}$ from \Cref{thm:selectable-bound} on the size of the Gramian based selectable set. 
\end{proposition}
\begin{proof}
To achieve this lower bound on the selectable set, we construct a sequence of rows to be selected, sampled. After the method selects each row in this sequence once, the size of the selectable set is exactly $\nrows -\abs{\indepset}$. Let $\indepset$ denote the maximum independent set of Gramian graph $G$. Then the sampling strategy in which we select each row from $\indepset$ in succession results in each of the rows of $\indepset$ being not selectable. Since none of the rows in $\indepset$ are neighbors in the graph, meaning that all rows in $\indepset$ are pairwise orthogonal, selecting any row in $\indepset$ does not impact the selectable or unselectable status of any other row in $\indepset$. Thus, if we select each row in $\indepset$ once, there will be exactly $\abs{\indepset}$ unselectable rows, meaning that size of the selectable set will be exactly $\nrows-\abs{\indepset}$. 
\end{proof}

 If the maximum independent set formed from the Gramian is large, then the size of the selectable set will be small for  \Cref{algo:gssrk}. 
 Thus, based on our convergence analysis for selectable set methods, \Cref{thm:selectable-one-step}, we will have a convergence guarantee improvement over Kaczmarz methods that do not use a selectable set. We examine and provide examples of this convergence improvement experimentally in \Cref{sec:experiments}.

Next, we investigate how we can apply \Cref{thm:selectable-bound} to basic structured non-orthogonality graphs and examine the size of their selectable sets. 
We select the following graphs as they correspond to non-orthogonality graphs constructed from very sparse systems, and many serve as simple but motivating examples for such structures that arise in practice. We anticipate that by examining these specific structured graphs we may be able to understand when applying GSSRK to a system will yield significant convergence improvement.

First, we consider the path graph. This is a graph in which the adjacency matrix \edit{(the Gramian)} has nonzero entries exclusively along the diagonal, the super diagonal \edit{and sub-diagonal}, representative of sparse matrices from networks applications.


\begin{corollary}[Path Gramian]
\label{cor:path}
For the case of a path graph Gramian with $\nrows$ vertices, the size of the selectable set is lower bounded by $\left \lfloor \frac{\nrows}{2} \right \rfloor$.
\end{corollary}
\begin{proof}
    Taking every other vertex in the path, we obtain a maximum independent set of size $\left \lceil \frac{\nrows}{2} \right \rceil$. We then apply \Cref{thm:selectable-bound} to obtain the bound $\abs{\selectset_{k}} \geq \nrows - \left \lceil \frac{\nrows}{2} \right \rceil =\left \lfloor \frac{\nrows}{2} \right \rfloor$.
\end{proof}

Next we consider a star graph whose adjacency matrix contains nonzeros on the diagonal, 
along a single row $\ell$, and the column $\ell$. This is representative of a matrix that is mostly sparse but may contain some nonsparse rows.
An example of this structure arises in collaborative filtering where data may represent user prescribed ratings. Most users will have few ratings representing sparse rows, but some super-users will have many ratings, representing the nonsparse rows of the matrix.  

\begin{corollary}[Star Gramian]
\label{cor:star}
For the case of a star graph Gramian with $\nrows$ vertices, the size of the selectable set is lower bounded by $1$.
\end{corollary} 

\begin{proof}
The $\nrows-1$ leaves form a maximum independent set of size $\nrows-1$. We then apply \Cref{thm:selectable-bound} to obtain the bound $\abs{\selectset_{k}} \geq \nrows - (\nrows-1)  = 1$.
\end{proof}

We note that the high degree node of the star represents a nonsparse row which may cause minimal convergence improvement when applying selectable set methods.
Next we consider a cycle graph in which the adjacency matrix has nonzeros along the diagonal, super diagonal, \edit{sub-diagonal}, \edit{top right corner} and the bottom left corner. This represents a sparse matrix similar to \Cref{cor:path} with slightly less sparsity. 

\begin{corollary}[Cycle Gramian]
\label{cor:cycle}
For the case of a cycle graph Gramian with $\nrows$ vertices, the size of the selectable set is lower bounded by $\left \lceil \frac{\nrows}{2} \right \rceil$.
\end{corollary}

\begin{proof}
Taking every other vertex in the cycle except the last one if $\nrows$ is odd, we obtain a maximum independent set of size $\left \lfloor \frac{\nrows}{2} \right \rfloor$. We then apply \Cref{thm:selectable-bound} to obtain the bound $\abs{\selectset_{k}} \geq \edit{\nrows} - \left \lfloor \frac{\nrows}{2} \right \rfloor = \left \lceil \frac{\nrows}{2} \right \rceil$.
\end{proof}

Next we consider a banded graph in which the adjacency matrix contains nonzeros along the diagonal, $\edit{\ell}$ super diagonals and $\edit{\ell}$ sub-diagonals. This adjacency matrix occurs commonly in semi-supervised graph learning tasks in which only data for the k nearest neighbors of a node is kept.


\begin{corollary}[Banded Gramian]
\label{cor:banded}
For the case of a banded matrix Gramian with $\nrows$ vertices, \edit{bandwidth $\ell$} \edit{(upper and lower bandwidth $\edit{\ell}$)}, the size of the selectable set is lower bounded by $\edit{\left \lfloor \frac{\ell \nrows}{\ell + 1} \right \rfloor}$.
\end{corollary}
\begin{proof}
\edit{The maximal independent set can be constructed by taking the first node, node 1, which is adjacent to $\ell$ other nodes. Then taking node $1+\ell +1$ to avoid the $\ell$ neighbors of the first node. This node has at most $2\ell$ neighbors but shares $\ell$ of them with the first node. So now we can take node $1+ 2\ell +2$. Repeating this process we deduce that we can select one node from every $\ell + 1$ nodes. Resulting in a maximal independent set of size  $\left \lceil \frac{\nrows}{\ell+1} \right \rceil$.
We then apply \Cref{thm:selectable-bound} to obtain the bound  $\abs{\selectset_{k}} \geq \nrows - \left \lceil \frac{\nrows}{\ell+1} \right \rceil =  \left \lfloor \frac{\nrows\ell + \nrows}{\ell+1} - \frac{\nrows}{\ell+1} \right \rfloor = \left \lfloor \frac{\ell \nrows}{\ell + 1} \right \rfloor$.}
\end{proof}

Note that a path graph is an example of a banded matrix with \edit{bandwidth 1}. Finally, we consider \edit{a symmetric} $\ell$ regular graph. The adjacency matrix for this graph has $\ell$ nonzero entries \edit{on each row and corresponding column}, excluding the diagonal entries \edit{thus $\ell < \nrows$}. \edit{Since this adjacency matrix corresponds to an undirected graph, it is symmetric.}
This adjacency matrix pattern is likely to occur when applying a generalized k-nearest neighbors algorithm to graph data.

\begin{corollary}[$\ell$-regular Gramian]
\label{cor:regular}
For the case of an $\ell$-regular graph with $\nrows$ vertices and degree $\ell \edit{ < m} $, the size of the selectable set is lower bounded by $\max(\left \lceil \frac{\nrows}{2} \right \rceil, \ell)$.
\end{corollary}
\begin{proof}
    The size of the maximum independent set of an $\ell$-regular graph \cite{rosenfeld1964independent} is upper bounded by \\ $\min(\left \lfloor \frac{\nrows}{2} \right \rfloor, \nrows-\ell)$. Thus, we can apply \Cref{thm:selectable-bound} to obtain the bound $\abs{\selectset_{k}} \geq \max(\left \lceil \frac{\nrows}{2}\right \rceil, \ell)$.
\end{proof}

\Cref{thm:selectable-bound} gives us a lower bound on the size of the selectable set in many cases, see \Cref{tab:ss-lowerbound}. However, to apply this theorem directly, we need an upper bound on the size of the maximum independent set. For some classes of graphs, while a lower bound may be achievable, the upper bound is the trivial bound of $\nrows - 1$. 

\begin{table}
\begin{center}
\begin{tabular}{lcc} 
 \toprule
 Gramian graph & \makecell{Maximum independent \\ set size} & \makecell{Lower bound on \\ size of selectable set $\abs{\selectset}$} \\ \midrule
path graph  & $\left \lceil \frac{\nrows}{2} \right \rceil$ & $\left \lfloor \frac{\nrows}{2} \right \rfloor$\\
star graph & m-1 & 1 \\
cycle graph & $\left \lfloor \frac{\nrows}{2} \right \rfloor$ & $\left \lceil \frac{\nrows}{2} \right \rceil$\\
banded matrix graph, \edit{bandwidth $\ell$} & $\edit{\left \lceil \frac{\nrows}{\ell+1} \right \rceil}$ & $\edit{\left \lfloor \frac{ \ell \nrows}{\ell + 1} \right \rfloor}$ \\
$\ell$-regular graph & $\min(\left \lfloor \frac{\nrows}{2}\right \rfloor, \nrows-\ell)$ & $\max(\left \lceil \frac{\nrows}{2} \right \rceil, \ell)$\\
 \bottomrule
\end{tabular}
\caption{Lower bounds on size of the selectable set for structured Gramian problems.}
\label{tab:ss-lowerbound}
\end{center}
\end{table}

\FloatBarrier
\section{Experiments}\label{sec:experiments}

We evaluate RK, NSSRK, GSSRK, and GRK on a series of synthetic and real world matrices from the  SuiteSparse matrix library \footnote{Implementation of methods used: \url{https://github.com/jdmoorman/kaczmarz-algorithms/}.}. For each matrix, we plot the squared error norm versus iteration number for 20000 iterations. Results were averaged over \edit{100} trials. \edit{Each line corresponds to the average error at the iteration over the 100 trials and the shading corresponds to one standard deviation above the mean and one standard deviation below the mean at each iteration.}\edit{The vectors $\xopt$ and $\vb$ for all of the experiments are constructed by taking a standard random normal vector $\vv$ (of mean 0 and standard deviation 1 entries), computing $\xopt = \frac{\mA^T \vv}{\norm{\mA^T \vv}}$, then applying $\mA$ such that $\vb = \mA \xopt$.} 

\edit{The two synthetic matrices that we consider are the circulant matrix and the 3-banded matrix.} 
\edit{The circulant matrix is a $100 \times 100 $ matrix with $\sigma_{\min} = 0.690$, see \Cref{fig:ssrk_circ_uniform} and \Cref{fig:ssrk_circ_row}. The matrix has non-zeros on the diagonal, sub-diagonal and top right corner. The non-zero entries of each row are the row number $i$ divided by $\sqrt{2}$ so each row has non-zero entries $\frac{i}{\sqrt{2}}$. The Gramian of the circulant matrix corresponds to the cycle graph Gramian described in \Cref{cor:cycle}.} 
\edit{The 3-banded matrix is a $100 \times 100$ matrix with standard random normal entries along the diagonal and the three rows above and below the diagonal producing a matrix with bandwidth 3 and $\sigma_{\min} = 0.0180$, see \Cref{fig:ssrk_banded_uniform} and \Cref{fig:ssrk_banded_row}. The Gramian of this matrix corresponds to a banded matrix graph with bandwidth 6, in \Cref{cor:banded}.} 

\edit{The two real world matrices we consider are Cities and the transpose of the N\_pid from the SuiteSparse matrix library \cite{davis2011university}.} 
\edit{The Cities matrix is a dense matrix of size $55 \times 46$ with $\sigma_{\min}= 0.271$, see \Cref{fig:ssrk_cities_uniform} and \Cref{fig:ssrk_cities_row}.}
\edit{The transpose of the N\_pid matrix is a sparse matrix of size $3923 \times 3625$ with $8054$ nonzero entries that has $\sigma_{\min}= 0.0690$, see \Cref{fig:ssrk_npid_uniform} and \Cref{fig:ssrk_npid_row}.}

\begin{figure}[ht]
    \centering
    \hfill
     \begin{subfigure}[b]{0.45\textwidth}
         \centering
         \includegraphics[width=\textwidth]{ 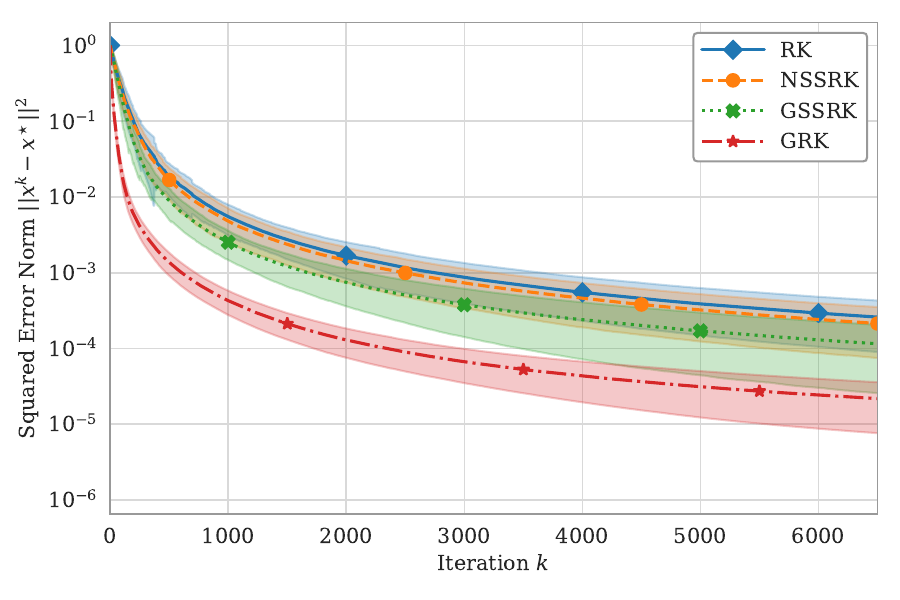}
         \captionsetup{font=small}
         \caption{Circulant Matrix Uniform Distribution $\sigma_{\min} = 0.690$}
         \label{fig:ssrk_circ_uniform}
     \end{subfigure}
    \hfill
    \begin{subfigure}[b]{0.45\textwidth}
         \centering
         \includegraphics[width=\textwidth]{ 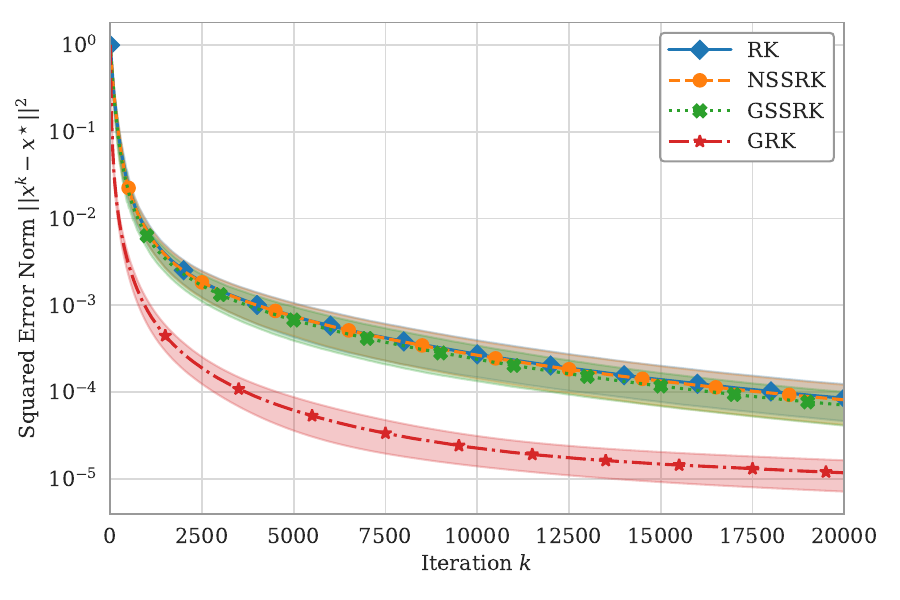}
         \captionsetup{font=small}
         \caption{Banded Matrix Uniform Distribution $\sigma_{\min} = 0.0180$}
         \label{fig:ssrk_banded_uniform}
     \end{subfigure}
   \vskip\baselineskip
     \hfill
     \begin{subfigure}[b]{0.45\textwidth}
         \centering
         \includegraphics[width=\textwidth]{ 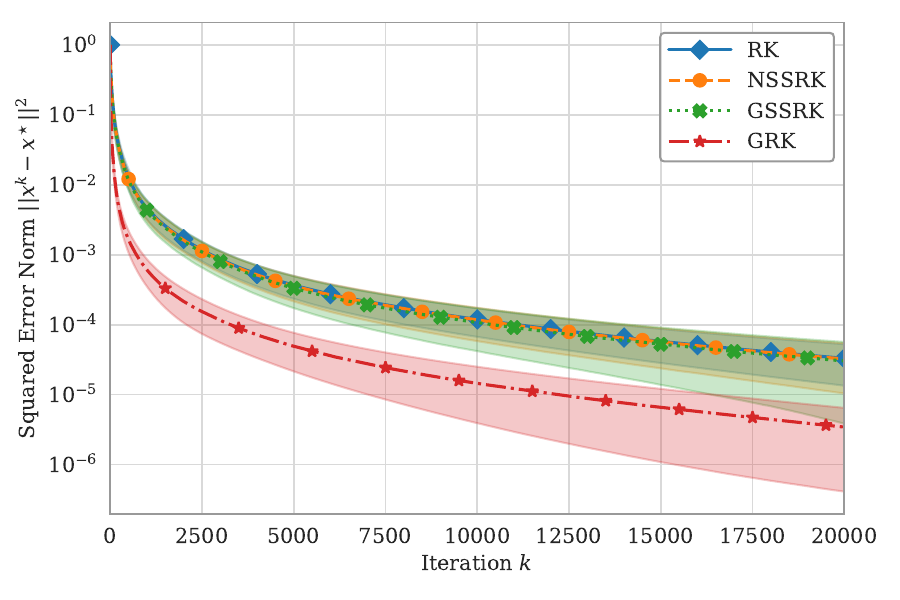}     \captionsetup{font=small}
         \caption{Cities Matrix Uniform Distribution $\sigma_{\min}= 0.271$}
         \label{fig:ssrk_cities_uniform}
     \end{subfigure}
     \hfill
     \begin{subfigure}[b]{0.45\textwidth}
         \centering
         \includegraphics[width=\textwidth]{ 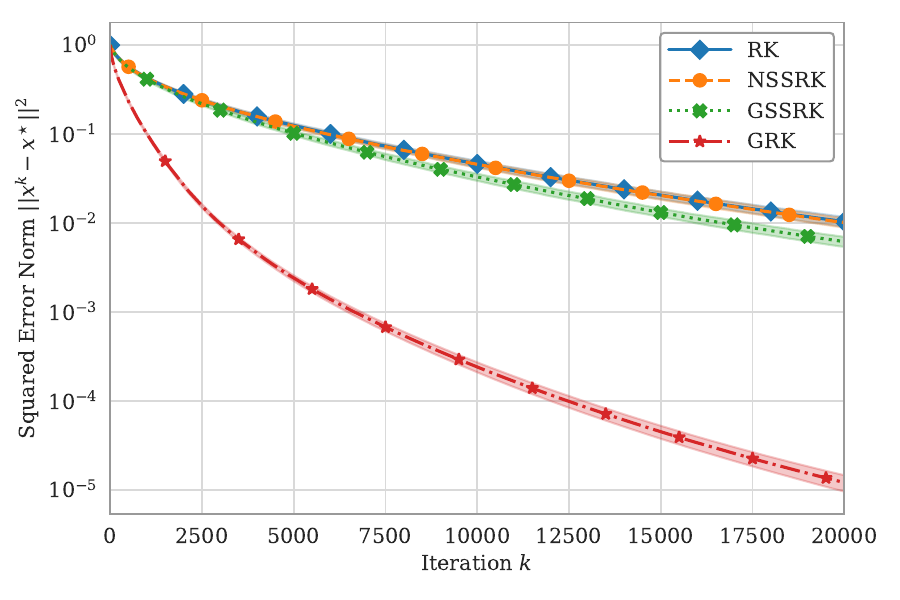}
         \captionsetup{font=small}
         \caption{N\_Pid Matrix Uniform Distribution $\sigma_{\min}= 0.0690$}
         \label{fig:ssrk_npid_uniform}
     \end{subfigure}
    \caption{\edit{Squared error norm versus iteration for RK, NSSRK, GSSRK, and GRK (RGRK with $\theta = 1/2$) using \emph{uniform row probabilities} for RK, NSSRK and GSSRK. Results were averaged over 100 trials. \Cref{fig:ssrk_banded_uniform} is a 3-banded matrix, \Cref{fig:ssrk_circ_uniform} is a circulant matrix, and \Cref{fig:ssrk_cities_uniform,fig:ssrk_npid_uniform} are two real world matrices, Cities and N\_pid all described in detail above. The shading denotes one standard deviation.}} 
    \label{fig:ssrk_figure_uniform}
\end{figure}

\edit{The four algorithms that we consider are RK, NSSRK, GSSRK, and GRK, each of which has differing computational complexity. RK has a computational complexity of $\mathcal{O}(\ncols)$ per iteration, where $\ncols$ is the number of columns in the matrix. Similarly, NSSSRK also has a computational complexity of $\mathcal{O}(\ncols)$, requiring $\mathcal{O}(1)$ to update the selectable set. Unlike NSSRK, GSSRK has the added overhead of updating the selectable set at each iteration requiring $\mathcal{O}(\ncols + \nrows)$ at each iteration, $\mathcal{O}(\ncols)$ to compute the Kaczmarz update, \Cref{eqn:kacz-update}, and $\mathcal{O}(\nrows)$, where $\nrows$ is the number of rows, to look at a row of the Gramian matrix and update the selectable set. Additionally, if it is necessary to pre-compute the Gramian with standard matrix multiplication, this requires $\mathcal{O}(\ncols\nrows^2)$. Finally, the most computationally costly method per iteration is GRK, which requires $\mathcal{O}(\ncols\nrows)$ at each iteration as a full residual computation, $\mA \vx_k - \vb$, is required to construct the sampling set and distribution. Of course, these are all generic bounds, and different implementations may lead to improvements, such as parallelization, fast multiplies, and other specific uses of the system's structure.}

In all four of the matrices pictured, as well as all the other matrices, we evaluated our methods on, we saw little to no difference between the performance of RK and NSSRK. GRK also outperformed RK, NSSRK, and GSSRK on every matrix. However, we did observe some differences in the performance of GSSRK relative to the other methods. In some examples, we see that GSSRK performs the same as RK and NSSRK, while in others, it performs slightly better.

In \Cref{fig:ssrk_banded_uniform} and \Cref{fig:ssrk_circ_uniform}, our two synthetic matrices, we see that one of them (\Cref{fig:ssrk_banded_uniform}) displays very little difference in performance between GSSRK and RK, while the other (\Cref{fig:ssrk_circ_uniform}) has a much larger gap in performance between the two. This is due to the difference in size of their respective selectable sets. In the circulant matrix, each row is only non-orthogonal to two others, while each row in the \edit{banded matrix} is non-orthogonal to six others. With a sparser non-orthogonality graph, we expect to see smaller selectable sets, which leads to improved performance.

We observe that in \Cref{fig:ssrk_cities_uniform} there is no performance gap between GSSRK and RK on the Cities matrix. Because the Cities matrix has no rows that are mutually orthogonal, the selectable set for this matrix includes every row, except for the one that was just picked, and thus the GSSRK method is equivalent to NSSRK. In \Cref{fig:ssrk_npid_uniform}, we see that GSSRK outperforms RK. Computing the Gramian of the N\_pid matrix, we see that it is much sparser, with fewer than $0.1\%$ of the entries being nonzero. Thus, the selectable set is much smaller, leading to improved performance, consistent with our theoretical results. \edit{We note that RK, NSSRK and GSSRK all require a sampling distribution that is commonly either the uniform or the row norm distribution \cite{strohmer2009randomized} while GRK requires a choice of $\theta$ that was set to 0.5 as in \cite{bai2018greedy}. The results from both uniform and row norm distributions are shown and both display similar comparisons, where GRK performs best but is most costly, followed by GSSRK, then NSSRK, then RK. }

\begin{figure}[ht]
    \centering
     \hfill
     \begin{subfigure}[b]{0.45\textwidth}
         \centering
         \includegraphics[width=\textwidth]{ 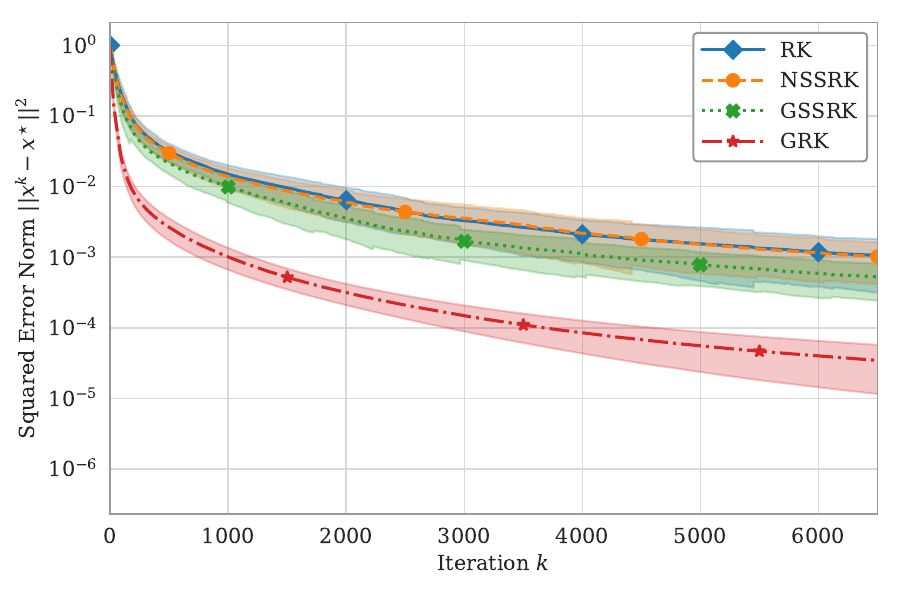}
         \captionsetup{font=small}
         \caption{Circulant Matrix Row Norm Distribution $\sigma_{\min} = 0.690$}
         \label{fig:ssrk_circ_row}
     \end{subfigure}
    \hfill
    \begin{subfigure}[b]{0.45\textwidth}
         \centering
         \includegraphics[width=\textwidth]{ 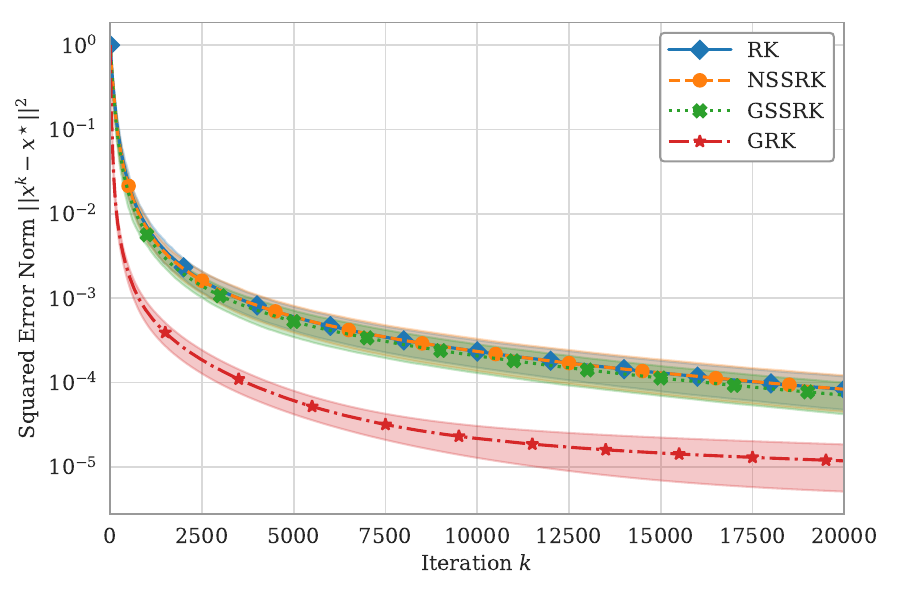}
         \captionsetup{font=small}
         \caption{Banded Matrix Row Norm Distribution $\sigma_{\min} = 0.0180$}
         \label{fig:ssrk_banded_row}
     \end{subfigure}
     \vskip\baselineskip
     \hfill
     \begin{subfigure}[b]{0.45\textwidth}
         \centering
         \includegraphics[width=\textwidth]{ 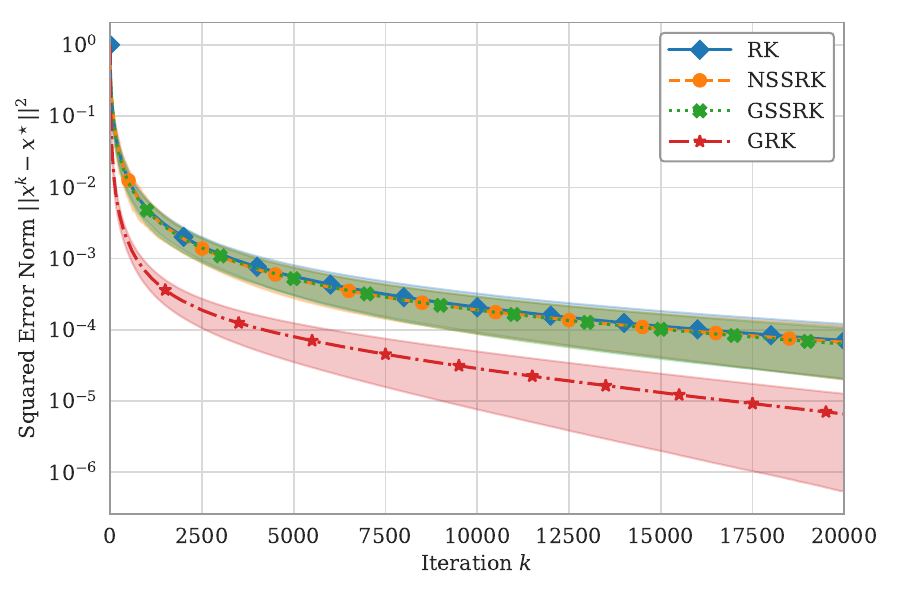}         \captionsetup{font=small}
         \caption{Cities Matrix Row Norm Distribution $\sigma_{\min}= 0.271$}
         \label{fig:ssrk_cities_row}
     \end{subfigure}
     \hfill
     \begin{subfigure}[b]{0.45\textwidth}
         \centering
         \includegraphics[width=\textwidth]{ 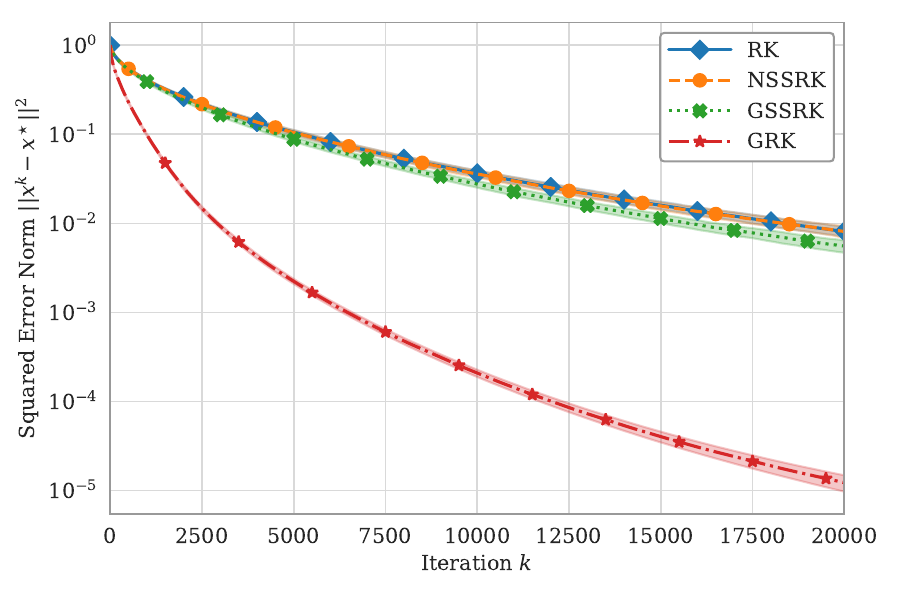}
         \captionsetup{font=small}
         \caption{N\_PID Matrix Row Norm Distribution $\sigma_{\min}= 0.0690$}
         \label{fig:ssrk_npid_row}
     \end{subfigure}
    \caption{\edit{Squared error norm versus iteration for RK, NSSRK, GSSRK, and GRK (RGRK with $\theta = 1/2$) using \emph{row norm probabilities} for RK, NSSRK and GSSRK. Results were averaged over 100 trials. \Cref{fig:ssrk_banded_row} is a 3-banded matrix, \Cref{fig:ssrk_circ_row} is a circulant matrix, and \Cref{fig:ssrk_cities_row,fig:ssrk_npid_row} are two real world matrices, Cities and N\_pid all described in detail above. The shading denotes one standard deviation.}} 
    \label{fig:ssrk_figure_row_norm}
\end{figure}

\edit{Next, we validate our theoretical results experimentally by varying $\sigma_{\min}$ and examining the convergence rates of GSSRK with \emph{uniform probabilities}, \Cref{algo:gssrk}. Since the SSRK convergence result, \Cref{thm:selectable-one-step}, depends on $\sigma_{\min}$, the smallest nonzero singular value, we construct three matrices with varying $\sigma_{\min}$ and observe their convergence rates. We construct circulant matrices of sizes $50 \times 50, 100 \times 100$ and $150 \times 150$ with ones along the diagonal, sub-diagonal and top right corner and corresponding $\sigma_{\min} = 0.126, \sigma_{\min} = 0.063$ and $\sigma_{\min} = 0.042$. The results shown in \Cref{fig:sig_min_experiment} support our theoretical guarantees. The results of the experiment on the circulant matrices were averaged over 100 trials for each matrix. The shading corresponds to one standard deviation above the mean and one standard deviation below the mean.}

\begin{figure}[ht]
    \centering
    \includegraphics[width=.5\textwidth]{ 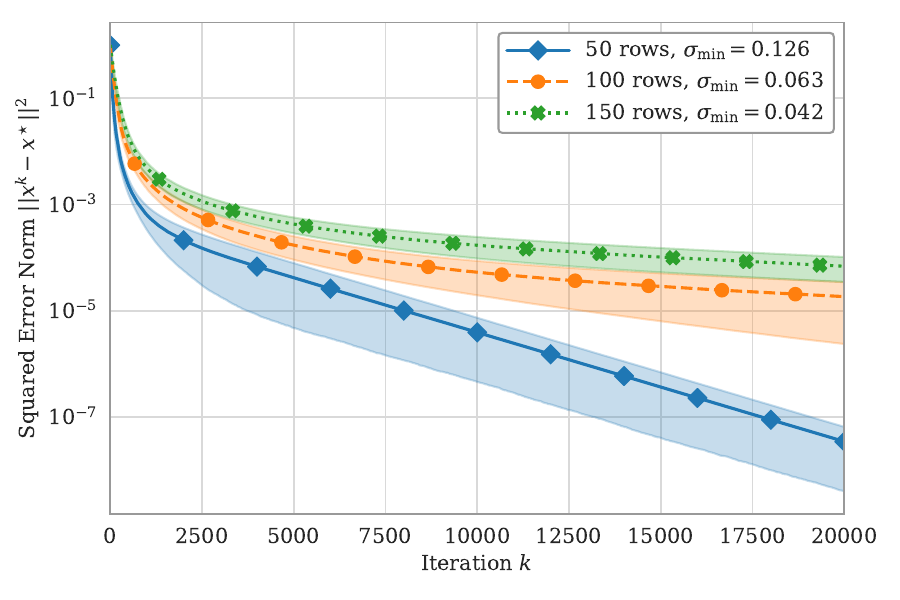} 
    \caption{\edit{Squared error norm versus iteration for GSSRK with uniform row distribution on varying circulant matrix sizes, corresponding to varying $\sigma_{\min}$, to confirm that as $\sigma_{\min}$ increases the convergence rate improves. 100 trials on circulant matrices of sizes: $50 \times 50, 100 \times 100 $ and $150 \times 150$. Each curve corresponds to the iteration average over the 100 trials and the shading corresponds to one standard deviation above and one standard deviation below the mean of norm-squared errors at each iteration.}}
    \label{fig:sig_min_experiment}
\end{figure}

\edit{In conclusion, our experiments \edit{support our theoretical bounds and we} observe that NSSRK and RK perform nearly identically. 
Since NSSRK has a minimal impact on the row sampling in comparison to RK methods, we do not expect improvement over RK specifically when $\nrows$ is large. 
Additionally we observe that GSSRK outperformed RK in some cases.
These cases include matrices with many orthogonal rows which is likely to occur when solving sparse systems. 
Even though we have improved convergence guarantees for both GSSRK and NSSRK over RK, we observe that the empirical performance of the algorithms may not reflect this.}
Finally, we note that GRK always outperforms GSSRK and NSSRK even though both methods have identical or improved theoretical guarantees over GRK \edit{when considering a $\sigma_{\min}(\mA)$  bound}. This supports the belief that there is a gap in the theoretical understanding of GRK and perhaps other Adaptive Kaczmarz methods \cite{gower2021adaptive}.

\FloatBarrier
\section{Conclusion}\label{sec:conclusion}
In this paper, we studied various selectable set approaches for the Kaczmarz method. 
These aimed to accelerate the standard approach of uniform random row sampling by attempting to move the iterates larger distances in each iteration, thereby reaching the desired solution in fewer steps. 
We proposed and analyzed a general selectable set randomized Kaczmarz method (\Cref{algo:selectablesetrk}) in which rows are sampled from a selectable set $\selectset$. 
The selectable set, $\selectset$, is updated at each iteration of the method and satisfies the condition that given a row of the matrix $j$, $j \not \in \selectset \implies \mA_j \xk = \vbj$. 
We proved a general convergence result for this method, \Cref{thm:selectable-bound}, in which we proved that there is an improvement inversely proportional to the size of the selectable set over Kaczmarz methods that sample from all rows. 

After we defined the general selectable set framework we examined several strategies such as the simple non-repetitive strategy (\Cref{algo:nssrk}) and the Gramian based strategy (\Cref{algo:gssrk}).
Although the non-repetitive strategy was quite simple and perhaps naive, it led to interesting theoretical results that are  comparable to those in state of the art methods such as the methods proposed in \cite{bai2018greedy,bai2018relaxed}, see \Cref{sec:bai-wu-compare}. This led us to believe that there is a theory gap in existing Kaczmarz literature and that a tighter analysis could be possible for some Kaczmarz methods. 

The Gramian based strategy leveraged the orthogonality structure of the matrix which is encoded in the Gramian, $\mG = \mA \mA^T$. There are many structured problems in which both the matrix $\mA$ and the Gramian $\mG$ are known such as graph-based semi-supervised learning \cite{bertozzi2018uncertainty}.
Leveraging the Gramian to update the selectable set yielded tighter convergence bounds than using the non-repetitive strategy, but with the added overhead of updating the selectable set using the Gramian. 
Bounds on GSSRK were originally proposed by Nutini et al. \cite{nutini2016convergence}; we generalized their results by constructing a bound in terms of the singular values of the original matrix, instead of
singular values of submatrices, which may be less natural to compute in applications.
We discussed the lower bound on the size of the selectable set for the Gramian update when dealing with structured problems, and noted that the size of the selectable set is usually on the order of number of rows $O(m)$. This relatively large size meant that there is a constant improvement of the convergence guarantee over corresponding non-selectable set methods.
Finally, we provided some numerical experiments on both synthetic and real-world matrices from the SuiteSparse package \cite{davis2011university} that demonstrated the benefits of using selectable sets and showcase how the variations compare. 

\subsection{Future directions}\label{sec:future}

We addressed the problem of projecting onto a row of an equation that has already been solved through sampling from a selectable set $\selectset$ of unsolved equations. \edit{In our methods, we consider the $p_i$ values as a fixed distribution that we renormalize based on the selectable set. Based on our one step convergence result, \Cref{thm:selectable-one-step}, one could try to optimize the probabilities $p_i$ themselves each iteration instead of drawing from this fixed distribution. This would require additional computational cost but may have the benefit of faster convergence rates.}

A \edit{different} simple yet surprisingly effective modification of the probability distribution at each iteration is \emph{sampling without replacement} \cite{patel2021implicit, recht2012toward}. 
In these methods a row can only be sampled again after all of the other rows of the matrix have been sampled. 
This contrasts the standard method of sampling \emph{with replacement} in which there is a fixed probability for a row to be selected at each iteration throughout the duration of the algorithm.
One possible explanation for why sampling without replacement shows empirical improvement over sampling with replacement could stem from the problem that selectable set methods aim to solve.
During some iterations of a Kaczmarz method, there is no improvement to the iterate because the method is projecting onto an already solved equation yielding no change to the iterate. However, empirically and theoretically, the selectable set method does not account for the improvement seen when sampling without replacement.  This suggests that there is more going on in the interaction between projections than simply being (near) orthogonal or sharing a solution space.

The Recht-R\'e conjecture \cite{recht2012toward} gave a theoretical explanation of why sampling without replacement outperforms sampling with replacement. Recent works 
\cite{lai2020recht, israel2016arithmetic} have proved that the Recht-R\'e conjecture does hold for smaller dimensions, giving a theoretical explanation of why sampling without replacement outperforms sampling with replacement. 
However, recently Lai et al. \cite{lai2020recht} were able to prove that for dimensions five or larger this conjecture is false. 
Thus the theory gap of proving why sampling without replacement outperforms  sampling with replacement in Kaczmarz methods is currently an open problem, that while seemingly related to selectable sets, is still not explained.

\edit{
\section*{Acknowledgements}
\edit{The authors are grateful to the reviewers and editor of this manuscript for their time and useful feedback. In addition, DN, WS and YY were partially supported by NSF DMS $\#2011140$. Additionally, YY was partially supported by NSF DMS $\#1737770$ and NSF DMS $\#2027277$.}
}

\bibliographystyle{unsrt}
\bibliography{main}
\end{document}